\documentclass[final,1p]{elsarticle}
\usepackage{fancyhdr, graphicx}
\usepackage{amsmath,amssymb,amsthm,mathrsfs,geometry}
\usepackage{latexsym}
\usepackage{color}
\usepackage{subfig}
\usepackage{algpseudocode}
\usepackage{algorithm}
\usepackage{epstopdf}
\usepackage{multirow}
\usepackage{cases}
\usepackage[colorlinks,linkcolor=blue,hyperindex]{hyperref}

\usepackage[pagewise]{lineno}
\newtheorem{theorem}{Theorem}

\newtheorem{remark}{Remark}
\newtheorem{corollary}{Corollary}
\newtheorem{definition}{Definition}

\numberwithin{equation}{section}
\usepackage{epsfig}
\usepackage{subfig}

\journal{ }
\begin{document}
\begin{frontmatter}

\title{Inverse Source Problem for a Time-Fractional Diffusion-Wave Equation with a Singular Inverse-Square Potential}

\author[gzmtu]{Zewen Wang}
\ead{zwwang6@163.com; wangzewen@gzmtu.edu.cn}

\author[NUIST]{Bin Wu}
\ead{binwu@nuist.edu.cn}

\author[KyotoU]{Yikan Liu}
\ead{liu.yikan.8z@kyoto-u.ac.jp}

\author[ecut]{Liangwei Jin}
\ead{3069374616@qq.com}

\author[gzmtu,gzmtuAI]{Shufang Qiu}
\ead{qiushufang@gzmtu.edu.cn}

\cortext[cor1]{Corresponding author.}

\address[gzmtu]{School of Arts and Sciences, Guangzhou Maritime University, Guangzhou, Guangdong, 510725, P.R. China}

\address[NUIST]{School of Mathematics and Statistics, Nanjing University of
Information Science and Technology, Nanjing 210044, P. R. China}

\address[KyotoU]{Department of Mathematics, Kyoto University, Kitashirakawa-Oiwakecho, Sakyo-ku, Kyoto 606-8502, Japan.}

\address[ecut]{School of Science, East China University of Technology, Nanchang, Jiangxi, 330013, P. R. China}

\address[gzmtuAI]{School of Artificial Intelligence, Guangzhou Maritime University, Guangzhou, Guangdong, 510725, China}

\begin{abstract}
This paper investigates an inverse source problem for a time-fractional diffusion-wave equation with a singular inverse-square potential. The source term is assumed to consist of a known temporal factor and an unknown spatial component, which is to be recovered from terminal-state measurements. The well-posedness and regularity of the forward problem are established within an appropriate energy framework by exploiting Hardy-type inequalities and the spectral properties of the associated singular elliptic operator. The terminal observation operator is then shown to be compact, and uniqueness of the spatial source is established under a suitable nondegeneracy condition on the temporal factor. To stabilize the resulting ill-posed inverse problem, a Tikhonov regularization approach is introduced. The gradient of the regularized functional is derived through an adjoint problem involving a right-sided fractional derivative, leading to an adjoint-based conjugate gradient method with an exact line search for the numerical reconstruction of the unknown source. Numerical experiments are conducted on both one‑ and two‑dimensional spatial domains, using both exact and noisy terminal data, to demonstrate the effectiveness and stability of the proposed source reconstruction method. 
\end{abstract}

\begin{keyword} 
Time-fractional diffusion-wave equation; inverse source problem; inverse-square potential; Tikhonov regularization;  conjugate gradient method.
\end{keyword}

\end{frontmatter}


\section{Introduction} \label{sec:introduction}

Fractional differential equations have become an important modeling tool for
evolution processes exhibiting memory, hereditary effects, or anomalous
transport. In contrast to classical diffusion and wave equations, the use of a
non-integer-order time derivative provides a unified framework that interpolates
between diffusive and wave-like dynamics. In particular, a time derivative of
order between zero and one leads to a subdiffusion model, whereas an order
between one and two gives rise to a fractional diffusion-wave equation
\cite{SchneiderWyss1989,SakamotoYamamoto2011}. The latter regime retains
long-memory effects while allowing propagation behavior that is not present in
standard subdiffusion. The analytical and numerical treatment of such equations
has attracted considerable attention; we refer, for example, to
\cite{SakamotoYamamoto2011,JinB2023} for the well-posedness, spectral
representation, regularity, and numerical analysis of time-fractional evolution
equations.

Inverse problems associated with fractional evolution equations arise when an
unknown coefficient, source, initial state, or fractional order must be
determined from indirect measurements. They are important in applications where
the internal forcing mechanism cannot be observed directly. At the same time,
they are typically ill-posed because the forward evolution smooths or weakens
part of the information carried by the unknown quantity. An overview of the
distinctive uniqueness, stability, and ill-posedness properties of inverse
problems for anomalous diffusion processes can be found in
\cite{JinRundell2015}.

Among the various inverse problems for time-fractional equations, the recovery
of a spatially dependent source from final-time data has been studied by means
of spectral regularization, Tikhonov minimization, quasi-boundary value
methods, and iterative algorithms. Jiang, Liu, and Wang
\cite{JiangD2020} formulated the reconstruction of the spatial component of a
separable source in a time-fractional diffusion equation as a stabilized
optimization problem and developed a fully discrete reconstruction method.
Wang et al.\ \cite{WangZ2023} proposed an exponential Tikhonov
regularization method and analyzed its convergence properties. For
fractional diffusion-wave equations, Wei and Luo
\cite{WeiT2022} introduced a generalized quasi-boundary value method for
recovering a spatial source, while Luo and Wei \cite{LuoY2024} subsequently
established uniqueness and developed a regularized numerical method based on a
final-time observation. These studies demonstrate that terminal measurements
can contain sufficient information for source identification, but also show
that regularization is indispensable for obtaining stable reconstructions.

The situation becomes more delicate when the spatial operator contains an
inverse-square potential. Such potentials represent a critical class of
singular lower-order terms and arise naturally in mathematical physics,
quantum mechanics, combustion models, and evolution equations with singular
geometry. The classical analysis of heat equations with singular potentials
was initiated by Baras and Goldstein \cite{BarasGoldstein1984}. The fundamental
role of the Hardy inequality and the critical threshold of the inverse-square
potential was clarified by Vázquez and Zuazua
\cite{VazquezJL2000}. A considerable literature has subsequently developed around direct, control, and inverse problems for parabolic equations with singular inverse-square
potentials. Vancostenoble \cite{VancostenobleJ2011} obtained Lipschitz stability
for an inverse source problem associated with a singular parabolic equation.
Cazacu \cite{CazacuC2014} studied controllability of the heat equation with a
boundary-localized inverse-square singularity. Related inverse source and
coefficient problems have been investigated for singular or degenerate
parabolic systems in
\cite{AlaouiM2021,AnhCT2022,QinX2023,QinX2025}. In particular, Qin and Li
\cite{QinX2025} established stability results for the recovery of a spatial
source in a singular parabolic equation with variable coefficients. An
optimal-control formulation and a Landweber-type reconstruction procedure for
a classical singular diffusion equation were recently considered in
\cite{NedjmaM2025}. From the numerical perspective, Ma and Chen
\cite{MaChen2024} developed an efficient finite difference and spectral
approximation for a time-fractional diffusion equation with an inverse-square
potential.

Despite these advances, most inverse source results for fractional
diffusion-wave equations have been obtained for spatial operators with regular
coefficients, whereas most studies involving inverse-square potentials concern
classical parabolic equations or fractional subdiffusion models. A particularly
relevant contribution is \cite{ChattouhA}, where the identification of a
spatial source in a time-fractional diffusion equation with an inverse-square
potential was analyzed in the subdiffusive regime. However, the extension to the diffusion-wave regime is not
straightforward. When the fractional order lies between one and two, two
initial conditions are required, the solution representation contains an
additional initial-velocity term, and the fractional integration-by-parts
formula generates a different set of terminal conditions for the adjoint
equation. Moreover, several positivity and monotonicity properties commonly
used in the subdiffusive case are no longer directly available.

Motivated by these observations, this paper investigates an inverse source problem for a time-fractional diffusion-wave equation with a singular inverse-square potential. The source is assumed to consist of a known term and a separable unknown component, whose temporal factor is prescribed and whose spatial factor is to be recovered from a terminal-state measurement. In contrast to the fractional subdiffusion setting, the diffusion-wave regime considered here involves two initial conditions and leads to a different adjoint structure in the inverse analysis. The singular coefficient is restricted to the subcritical Hardy regime, which provides the coercivity and spectral properties required for the analysis of the associated spatial operator.

The remainder of this paper is organized as follows. Section 2 introduces the mathematical model and presents the preliminary results and functional framework needed for the subsequent analysis. Section 3 is devoted to the forward problem, where its well-posedness and regularity are established. In Section 4, the inverse source problem is formulated and the uniqueness of the unknown spatial source is investigated. Section 5 presents the Tikhonov regularization framework and develops the corresponding conjugate gradient reconstruction method. Numerical experiments are reported in Section 6 to illustrate the accuracy and stability of the proposed reconstruction method. Finally, Section 7 concludes the paper.

\section{Model and Preliminaries}

\subsection{Problem setting}

Let $\Omega=(0,1)$ and $\Omega=\Omega\times(0,T)$ with $T>0$. We consider the following initial-boundary value problem for the time-fractional diffusion-wave equation:
\begin{equation}\label{eq:main}
\begin{cases}
{}_0^C D_t^\alpha u(x,t)
-u_{xx}(x,t)
-\dfrac{\mu}{x^2}u(x,t)
=p(x)q(t),
& (x,t)\in\Omega\times(0,T),\\[3pt]
u(x,0)=u_0(x),\quad
\partial_tu(x,0)=u_1(x),
& x\in\Omega,\\[3pt]
u(0,t)=u(1,t)=0,
& t\in(0,T),
\end{cases}
\end{equation}
where $1<\alpha<2$ and ${}_0^C D_t^\alpha$ denotes the left-sided Caputo fractional derivative of order $\alpha$, defined by
$$
{}_0^C D_t^\alpha u(t)
=
\frac{1}{\Gamma(2-\alpha)}
\int_0^t
(t-s)^{1-\alpha}u''(s)\,ds
=
I_{0+}^{2-\alpha}u''(t).
$$
Here, $I_{0+}^{2-\alpha}$ denotes the fractional integral of order $2-\alpha$ in the Riemann--Liouville sense.

When all parameters and functions appearing in \eqref{eq:main} are prescribed, solving for the unknown solution $u$ constitutes the \textbf{forward problem}. In what follows, however, we are concerned with an associated \textbf{inverse problem}, in which the temporal factor $q\in L^\infty(0,T)$ is assumed to be known a priori, while the spatial factor $p\in L^2(\Omega)$ is unknown and to be determined. As the observable data, we take the terminal state of the solution:
\begin{equation}\label{eq:obs}
u(x,T)=\omega_{\mathrm{obs}}(x),
\quad x\in\Omega.
\end{equation}
The inverse problem then consists in recovering the unknown spatial source $p$ from the terminal observation \eqref{eq:obs}. To define and analyze the associated source-to-terminal map, we first establish the well-posedness and regularity of the forward problem \eqref{eq:main} for a given source term $p\in L^2(\Omega)$.

\subsection{Hardy energy space and the singular elliptic operator}

Let $\mu^\ast=\frac{1}{4}$ be the critical constant in the one-dimensional Hardy inequality, and throughout this paper we assume  
$$
\mu<\mu^\ast.
$$
For $v,w\in H_0^1(\Omega)$, we define the symmetric bilinear form
\begin{equation}\label{eq:amu}
a_\mu(v,w)
:=
\int_0^1 v_x w_x \, dx
-
\mu\int_0^1\frac{vw}{x^2}\,dx .
\end{equation}
The one-dimensional Hardy inequality \cite{VazquezJL2000,SuD2012}
\begin{equation}\label{eq:hardy}
\int_0^1\frac{|v(x)|^2}{x^2}\,dx
\le
4\int_0^1|v_x(x)|^2\,dx,
\quad
v\in H_0^1(0,1),
\end{equation}
ensures that $a_\mu$ is both continuous and coercive on $H_0^1(\Omega)$. More precisely, there exist constants $c_\mu>0$ and $C_\mu>0$ such that
\begin{equation}\label{eq:form-equivalence}
c_\mu\|v\|_{H_0^1(\Omega)}^2
\le
a_\mu(v,v)
\le
C_\mu\|v\|_{H_0^1(\Omega)}^2,
\quad
v\in H_0^1(\Omega).
\end{equation}

Indeed, for all $v,w\in H_0^1(\Omega)$, applying the Hardy inequality yields
$$
\begin{aligned}
|a_\mu(v,w)|
&\le
\|v_x\|_{L^2}\|w_x\|_{L^2}
+
|\mu|
\left\|\frac{v}{x}\right\|_{L^2}
\left\|\frac{w}{x}\right\|_{L^2}
\\
&\le
(1+4|\mu|)
\|v_x\|_{L^2}
\|w_x\|_{L^2},
\end{aligned}
$$
which establishes the continuity of $a_\mu$ on $H_0^1(\Omega)\times H_0^1(\Omega)$. For the coercivity, note that
$$
a_\mu(v,v)
=
\|v_x\|_{L^2}^2
-
\mu
\left\|\frac{v}{x}\right\|_{L^2}^2 .
$$
If $0\le \mu<1/4$, then Hardy's inequality implies
$$
a_\mu(v,v)
\ge
(1-4\mu)\|v_x\|_{L^2}^2.
$$
If $\mu<0$, we have
$$
a_\mu(v,v)
=
\|v_x\|_{L^2}^2
+
|\mu|
\left\|\frac{v}{x}\right\|_{L^2}^2
\ge
\|v_x\|_{L^2}^2.
$$
Thus, for every $\mu<1/4$, there exists a constant $c_\mu>0$ such that
$$
a_\mu(v,v)
\ge
c_\mu\|v\|_{H_0^1(\Omega)}^2,
\quad
v\in H_0^1(\Omega).
$$

Consequently, we introduce the notation
$$
H_0^{1,\mu}(\Omega)
:=
H_0^1(\Omega),
$$
where the latter space is understood to be equipped with the equivalent Hardy energy norm
\begin{equation}\label{eq:mu-norm}
\|v\|_\mu
:=
a_\mu(v,v)^{1/2}.
\end{equation}
Equivalently, $H_0^{1,\mu}(\Omega)$ may be defined as the completion of $C_c^\infty(\Omega)$ with respect to the norm $\|\cdot\|_\mu$. In the subcritical regime $\mu<1/4$, this space coincides with $H_0^1(\Omega)$ as a set, and the two norms are equivalent.

Since $a_\mu$ is densely defined, symmetric, closed, and coercive on $L^2(\Omega)$, the representation theorem for closed quadratic forms ensures the existence of a unique positive self-adjoint operator $L_\mu$ in $L^2(\Omega)$, characterized by
\begin{equation}\label{eq:Lmu-domain}
D(L_\mu)\!
:=\!
\left\{
v\in H_0^{1,\mu}(\Omega):\!
\exists\,g\in L^2(\Omega)
\text{ such that }
a_\mu(v,\varphi)\!=\!(g,\varphi)_{L^2(\Omega)}
\ \text{for all }\varphi\in H_0^{1,\mu}(\Omega)
\right\},
\end{equation}
with $L_\mu v=g$. In the distributional sense, $L_\mu$ corresponds to the Friedrichs realization of the differential operator
$$
-\Delta-\frac{\mu}{x^2}
$$
subject to the homogeneous Dirichlet boundary condition.

Since the embedding
$$
H_0^{1,\mu}(\Omega)\hookrightarrow L^2(\Omega)
$$
is compact, the operator $L_\mu$ possesses compact resolvent. Consequently, there exists an orthonormal basis $\{\phi_n\}_{n\ge1}$ of $L^2(\Omega)$ consisting of eigenfunctions of $L_\mu$, with corresponding eigenvalues
$$
0<\lambda_1\le\lambda_2\le\cdots,
\quad
\lambda_n\to\infty,
$$
such that
\begin{equation}\label{eq:eigen}
L_\mu\phi_n=\lambda_n\phi_n.
\end{equation}

For $\sigma\ge0$, we define the Hilbert scale
$$
\mathcal H^\sigma(\Omega)
:=
D(L_\mu^{\sigma/2}),
\quad
\|v\|_{\mathcal H^\sigma}^2
=
\sum_{n=1}^{\infty}
\lambda_n^\sigma
|(v,\phi_n)_{L^2(\Omega)}|^2 .
$$
For $\sigma<0$, the space $\mathcal H^\sigma(\Omega)$ is defined as the completion of $L^2(\Omega)$ with respect to the norm
$$
\|v\|_{\mathcal H^\sigma}^2
=
\sum_{n=1}^{\infty}
\lambda_n^\sigma
|(v,\phi_n)_{L^2(\Omega)}|^2,
$$
or equivalently, via duality, $\mathcal H^\sigma(\Omega)=\bigl(\mathcal H^{-\sigma}(\Omega)\bigr)'$ with $L^2(\Omega)$ serving as the pivot space. In particular, we recover the identities
$$
\mathcal H^1(\Omega)=H_0^{1,\mu}(\Omega),
\quad
\mathcal H^{-1}(\Omega)=\bigl(H_0^{1,\mu}(\Omega)\bigr)',
$$
and the norm equivalence
\begin{equation}\label{eq:H1-mu}
\|v\|_{\mathcal H^1(\Omega)}^2
=
\sum_{n=1}^{\infty}
\lambda_n|(v,\phi_n)_{L^2(\Omega)}|^2
=
a_\mu(v,v)
=
\|v\|_\mu^2.
\end{equation}

\section{Forward Problem: Well-Posedness and Regularity}

For a fixed $p\in L^2(\Omega)$, set
\begin{equation}\label{eq:def-f}
f(t):=p\,q(t).
\end{equation}
Since $q\in L^\infty(0,T)$,
$$
f\in L^\infty(0,T;L^2(\Omega))
\subset L^2(0,T;L^2(\Omega)),
$$
and
\begin{equation}\label{eq:f-estimate}
\|f\|_{L^2(0,T;L^2(\Omega))}
\le
T^{1/2}
\|p\|_{L^2(\Omega)}
\|q\|_{L^\infty(0,T)}.
\end{equation}

\subsection{Weak solution and spectral representation}

We first specify the solution concept used throughout the paper.
The equation is understood weakly only with respect to the spatial
variable; no global space--time variational reformulation is required.

\begin{definition}[Weak solution]\label{def:weak-solution}
Let $u_0\in\mathcal H^1(\Omega)$, $u_1\in L^2(\Omega)$. 
A function $u$ is called a weak solution of \eqref{eq:main} if
$$
u\in C([0,T];\mathcal H^1(\Omega))
\cap
C^1([0,T];\mathcal H^{-1}(\Omega)),
\quad 
{}_0^C D_t^\alpha u
\in
L^2(0,T;\mathcal H^{-1}(\Omega)),
$$
and
$$
u(\cdot,0)=u_0
\quad\text{in }\mathcal H^1(\Omega),
\quad
\partial_tu(\cdot,0)=u_1
\quad\text{in }\mathcal H^{-1}(\Omega).
$$
Moreover, for a.e. $t\in(0,T)$,
\begin{equation}\label{eq:weak-form}
\left\langle
{}_0^C D_t^\alpha u(t),v
\right\rangle_{\mathcal H^{-1},\mathcal H^1}
+
a_\mu(u(t),v)
=
(f(t),v)_H,
\quad
\forall v\in\mathcal H^1(\Omega).
\end{equation}
\end{definition}

The identity \eqref{eq:weak-form} is the spatial weak form of the
evolution equation. It is used to characterize the solution, whereas
existence and regularity will be established below by the spectral
representation of $L_\mu$.

For
$$
u_{0,n}:=(u_0,\phi_n)_{L^2},
\quad
u_{1,n}:=(u_1,\phi_n)_{L^2},
\quad
f_n(t):=(f(t),\phi_n)_{L^2},
$$
the $n$th Fourier coefficient of a solution satisfies
\begin{equation}\label{eq:modal}
{}_0^C D_t^\alpha u_n(t)
+
\lambda_nu_n(t)
=
f_n(t),
\end{equation}
with $u_n(0)=u_{0,n}$, $u_n'(0)=u_{1,n}$. The variation-of-constants formula gives
\begin{equation}\label{eq:modal-repr}
\begin{aligned}
u_n(t)
={}&
u_{0,n}
E_{\alpha,1}(-\lambda_nt^\alpha)
+
u_{1,n}
tE_{\alpha,2}(-\lambda_nt^\alpha)
\\
&+
\int_0^t
(t-s)^{\alpha-1}
E_{\alpha,\alpha}
\bigl(-\lambda_n(t-s)^\alpha\bigr)
f_n(s)\,ds .
\end{aligned}
\end{equation}
Accordingly, we introduce
\begin{equation}\label{eq:solution-repr}
\begin{aligned}
u(t)
={}&
E_{\alpha,1}(-t^\alpha L_\mu)u_0
+
tE_{\alpha,2}(-t^\alpha L_\mu)u_1
\\
&+
\int_0^t
(t-s)^{\alpha-1}
E_{\alpha,\alpha}
\bigl(-(t-s)^\alpha L_\mu\bigr)
f(s)\,ds ,
\end{aligned}
\end{equation}
where the operator-valued Mittag--Leffler functions are defined
spectrally. For instance,
$$
E_{\alpha,\beta}(-t^\alpha L_\mu)v
=
\sum_{n=1}^{\infty}
E_{\alpha,\beta}(-\lambda_nt^\alpha)
(v,\phi_n)_{L^2}\phi_n.
$$

This representation is the basic tool for both the analysis of the
forward problem and the subsequent inverse source problem.

\subsection{Existence, uniqueness, and regularity}

\begin{theorem}[Well-posedness]\label{thm:wellposed}
Let $1<\alpha<2$, $\mu<\frac{1}{4}$, $p\in L^2(\Omega)$, $q\in L^\infty(0,T)$, and assume
$u_0\in\mathcal H^1(\Omega)$, $u_1\in L^2(\Omega)$. 
Then problem \eqref{eq:main} admits a unique weak solution in the sense
of Definition~\ref{def:weak-solution}. In particular,
\begin{equation}\label{eq:regularity-u}
u\in
C([0,T];\mathcal H^1(\Omega))
\cap
C^1([0,T];\mathcal H^{-1}(\Omega)),
\end{equation}
and
\begin{equation}\label{eq:regularity-caputo}
{}_0^C D_t^\alpha u
\in
L^2(0,T;\mathcal H^{-1}(\Omega)).
\end{equation}
Moreover,
\begin{equation}\label{eq:apriori-strong}
\begin{aligned}
&
\|u\|_{C([0,T];\mathcal H^1)}
+
\|\partial_tu\|_{C([0,T];\mathcal H^{-1})}
+
\|{}_0^C D_t^\alpha u\|_
{L^2(0,T;\mathcal H^{-1})}
\\
&\quad
\le
C
\left(
\|u_0\|_{\mathcal H^1}
+
\|u_1\|_{L^2}
+
\|p\|_{L^2(\Omega)}
\|q\|_{L^\infty(0,T)}
\right),
\end{aligned}
\end{equation}
where $C>0$ depends only on
$\alpha$, $T$, $\mu$, and $\Omega$.
\end{theorem}

\begin{proof}
We construct the solution by \eqref{eq:solution-repr}.
For $1<\alpha<2$, the standard Mittag--Leffler estimates \cite{SakamotoYamamoto2011,KilbasAA2006} imply
$$
|E_{\alpha,\beta}(-r)|
\le
\frac{C}{1+r},
\quad
r\ge0,
$$
for the parameters appearing below.

First, consider the initial displacement term. Since $|E_{\alpha,1}(-\lambda_n t^\alpha)|\le C$, 
we obtain
$$
\sum_{n=1}^{\infty}
\lambda_n
| u_{0,n}E_{\alpha,1}(-\lambda_n t^\alpha)|^2
\le
C
\sum_{n=1}^{\infty}
\lambda_n| u_{0,n}|^2
=
C\| u_0\|_{\mathcal H^1(\Omega)}^2.
$$
Therefore,
$$
E_{\alpha,1}(-t^\alpha L_\mu) u_0:=
\sum_{n=1}^{\infty}
E_{\alpha,1}(-t^\alpha\lambda_n) u_{0,n}\phi_n
\in
L^\infty(0,T;\mathcal H^1(\Omega)).
$$
Moreover, for each $n$,
$E_{\alpha,1}(-\lambda_n t^\alpha)$ is continuous in $t$.
The previous bound provides a summable majorant, so by dominated
convergence,
$$
E_{\alpha,1}(-t^\alpha L_\mu) u_0
\in
C([0,T];\mathcal H^1(\Omega)).
$$

Next, consider the initial velocity term. By the Mittag--Leffler estimate,
$$
\lambda_n^{1/2}
\left|
tE_{\alpha,2}(-\lambda_n t^\alpha)
\right|
\le
C
\frac{\lambda_n^{1/2}t}{1+\lambda_n t^\alpha}.
$$
Since
$$
\sup_{\lambda>0}
\frac{\lambda^{1/2}t}{1+\lambda t^\alpha}
\le
Ct^{1-\alpha/2},
$$
we get
$$
\left\|
tE_{\alpha,2}(-t^\alpha L_\mu) u_1
\right\|_{\mathcal H^1(\Omega)}
\le
Ct^{1-\alpha/2}\| u_1\|_{L^2(\Omega)}.
$$
Because $1-\alpha/2>0$, this term tends to zero in
$\mathcal H^1(\Omega)$ as $t\to0^+$. Continuity for $t>0$ follows again
from dominated convergence on compact subintervals of $(0,T]$. Hence
$$
tE_{\alpha,2}(-t^\alpha L_\mu) u_1
\in
C([0,T];\mathcal H^1(\Omega)).
$$

Now consider the source term
$$
G(t)
:=
\int_0^t
(t-s)^{\alpha-1}
E_{\alpha,\alpha}(-(t-s)^\alpha L_\mu)f(s)\,ds .
$$
Define the operator kernel
$$
K(t):=
t^{\alpha-1}E_{\alpha,\alpha}(-t^\alpha L_\mu),
\quad
t>0.
$$
For $z\in L^2(\Omega)$, using the Mittag--Leffler estimate, we have
$$
\begin{aligned}
\|K(t)z\|_{\mathcal H^1}^2
&=
\sum_{n=1}^{\infty}
\lambda_n
\left|
t^{\alpha-1}
E_{\alpha,\alpha}(-\lambda_n t^\alpha)
(z,\phi_n)_{L^2}
\right|^2
\\
&\le
C
\left(
\sup_{\lambda>0}
\frac{\lambda^{1/2}t^{\alpha-1}}{1+\lambda t^\alpha}
\right)^2
\|z\|_{L^2}^2 .
\end{aligned}
$$
Since
$$
\sup_{\lambda>0}
\frac{\lambda^{1/2}t^{\alpha-1}}{1+\lambda t^\alpha}
\le
Ct^{\alpha/2-1},
$$
we get
$$
\|K(t)z\|_{\mathcal H^1}
\le
Ct^{\alpha/2-1}\|z\|_{L^2}.
$$
Because $\alpha/2-1\in(-1/2,0)$, the function
$t^{\alpha/2-1}$ belongs to $L^1(0,T)$. Therefore,
$$
K\in L^1(0,T;\mathcal L(L^2,\mathcal H^1)).
$$
Noting that $G(t)=\int_0^t K(t-s)f(s)\,ds$, 
we have
$$
\begin{aligned}
\|G(t)\|_{\mathcal H^1(\Omega)}
&\le
\int_0^t
\|K(t-s)f(s)\|_{\mathcal H^1(\Omega)}\,ds
\\
&\le
C\int_0^t
(t-s)^{\alpha/2-1}
\|f(s)\|_{L^2}\,ds
\\
&\le
C\|f\|_{L^\infty(0,T;L^2)}
\int_0^t
(t-s)^{\alpha/2-1}\,ds
\\
&=
C t^{\alpha/2}
\|f\|_{L^\infty(0,T;L^2)}.
\end{aligned}
$$
Therefore,
$$
\|G\|_{C([0,T];\mathcal H^1(\Omega))}
\le
C_T
\|f\|_{L^\infty(0,T;L^2)}\le
C
\|p\|_{L^2(\Omega)}
\|q\|_{L^\infty(0,T)}.
$$
Moreover, since
$$
\|G(t)\|_{\mathcal H^1(\Omega)}
\le
C t^{\alpha/2}
\|f\|_{L^\infty(0,T;L^2(\Omega))}
\longrightarrow 0
\quad
\text{as }t\to0^+,
$$
we have $G(0)=0$. By the continuity property of Bochner convolutions with an
$L^1$-kernel, it follows that
$$
G\in C([0,T];\mathcal H^1(\Omega)).
$$
It follows that
$$
u\in C([0,T];\mathcal H^1),
\quad
u(0)=u_0.
$$

Differentiating the spectral representation gives
$$
\frac{d}{dt}
E_{\alpha,1}(-\lambda t^\alpha)
=
-\lambda t^{\alpha-1}
E_{\alpha,\alpha}(-\lambda t^\alpha)
$$
and
$$
\frac{d}{dt}
\left[
tE_{\alpha,2}(-\lambda t^\alpha)
\right]
=
E_{\alpha,1}(-\lambda t^\alpha).
$$
For the source term,
$$
\frac{d}{dt}
\left[
t^{\alpha-1}
E_{\alpha,\alpha}(-\lambda t^\alpha)
\right]
=
t^{\alpha-2}
E_{\alpha,\alpha-1}(-\lambda t^\alpha).
$$
Since
$$
t^{\alpha-2}\in L^1(0,T),
$$
the same convolution argument, now in
$\mathcal H^{-1}(\Omega)$, gives
$$
\partial_tu
\in
C([0,T];\mathcal H^{-1}(\Omega)),
\quad
\partial_tu(0)=u_1.
$$

We next prove that
$$
\partial_t u\in C([0,T];\mathcal H^{-1}(\Omega)).
$$
For the initial displacement term, we use
$$
\frac{d}{dt}
E_{\alpha,1}(-\lambda_n t^\alpha)
=
-\lambda_n t^{\alpha-1}
E_{\alpha,\alpha}(-\lambda_n t^\alpha).
$$
Therefore,
$$
\frac{d}{dt}
\left[
 u_{0,n}E_{\alpha,1}(-\lambda_n t^\alpha)
\right]
=
-\lambda_n u_{0,n}t^{\alpha-1}
E_{\alpha,\alpha}(-\lambda_n t^\alpha).
$$
Then
$$
\begin{aligned}
\sum_{n=1}^{\infty}
\lambda_n^{-1}
\left|
\lambda_n u_{0,n}
t^{\alpha-1}
E_{\alpha,\alpha}(-\lambda_n t^\alpha)
\right|^2
&\le
C t^{2\alpha-2}
\sum_{n=1}^{\infty}
\lambda_n| u_{0,n}|^2
=
C t^{2\alpha-2}
\| u_0\|_{\mathcal H^1}^2.
\end{aligned}
$$
Hence this derivative term belongs to
$C([0,T];\mathcal H^{-1}(\Omega))$ and vanishes at $t=0$.

For the initial velocity term, the identity
$$
\frac{d}{dt}
\left[
tE_{\alpha,2}(-\lambda_n t^\alpha)
\right]
=
E_{\alpha,1}(-\lambda_n t^\alpha)
$$
gives
$$
\frac{d}{dt}
\left[
 u_{1,n}tE_{\alpha,2}(-\lambda_n t^\alpha)
\right]
=
 u_{1,n}E_{\alpha,1}(-\lambda_n t^\alpha).
$$
Since
$$
\sum_{n=1}^{\infty}
\lambda_n^{-1}
\left|
 u_{1,n}E_{\alpha,1}(-\lambda_n t^\alpha)
\right|^2
\le
C
\sum_{n=1}^{\infty}
\lambda_n^{-1}| u_{1,n}|^2
\le
C\| u_1\|_{L^2(\Omega)}^2,
$$
dominated convergence implies
$$
E_{\alpha,1}(-t^\alpha L_\mu) u_1
\in
C([0,T];\mathcal H^{-1}(\Omega)).
$$
Moreover,
$$
E_{\alpha,1}(0) u_1= u_1
\quad
\text{in }\mathcal H^{-1}(\Omega).
$$

It remains to deal with the derivative of the source term. Using
$$
\frac{d}{dt}
\left[
t^{\alpha-1}E_{\alpha,\alpha}(-\lambda t^\alpha)
\right]
=
t^{\alpha-2}E_{\alpha,\alpha-1}(-\lambda t^\alpha),
$$
define
$$
K_1(t):=
t^{\alpha-2}E_{\alpha,\alpha-1}(-t^\alpha L_\mu),
\quad
t>0.
$$
For $z\in L^2$, we estimate $K_1(t)z$ in $\mathcal H^{-1}$:
$$
\begin{aligned}
\|K_1(t)z\|_{\mathcal H^{-1}}^2
&=
\sum_{n=1}^{\infty}
\lambda_n^{-1}
\left|
t^{\alpha-2}
E_{\alpha,\alpha-1}(-\lambda_n t^\alpha)
(z,\phi_n)_{L^2}
\right|^2
\\
&\le
C t^{2\alpha-4}
\sum_{n=1}^{\infty}
\lambda_n^{-1}
|(z,\phi_n)_{L^2}|^2
\\
&\le
C t^{2\alpha-4}\|z\|_{L^2}^2,
\end{aligned}
$$
where we used $\lambda_1>0$. Hence
$$
\|K_1(t)z\|_{\mathcal H^{-1}}
\le
Ct^{\alpha-2}\|z\|_{L^2}.
$$
Since $\alpha-2\in(-1,0)$, we have
$t^{\alpha-2}\in L^1(0,T)$. Therefore,
$$
K_1\in L^1(0,T;\mathcal L(L^2,\mathcal H^{-1})).
$$
Consequently,
$$
\frac{d}{dt}G(t)
=
\int_0^t K_1(t-s)f(s)\,ds
$$
belongs to $C([0,T];\mathcal H^{-1}(\Omega))$.
Moreover,
$$
\left\|
\frac{d}{dt}G(t)
\right\|_{\mathcal H^{-1}}
\le
C
\int_0^t
(t-s)^{\alpha-2}
\|f(s)\|_{L^2}\,ds
\le
C t^{\alpha-1}
\|f\|_{L^\infty(0,T;L^2)},
$$
so
$$
\frac{d}{dt}G(0)=0
\quad
\text{in }\mathcal H^{-1}(\Omega).
$$
Combining the three derivative estimates yields
$$
\partial_t u
\in
C([0,T];\mathcal H^{-1}(\Omega)),
$$
and
$$
\partial_t u(\cdot,0)= u_1
\quad
\text{in }\mathcal H^{-1}(\Omega).
$$

Moreover, since
$$
\|L_\mu u\|_{L^2(0,T;\mathcal H^{-1}(\Omega))}
=
\|u\|_{L^2(0,T;\mathcal H^1(\Omega))}, 
$$
and
$$\|u\|_{C([0,T];\mathcal H^1)}
\le
C
\left(
\|u_0\|_{\mathcal H^1}
+
\|u_1\|_{L^2}
+
\|p\|_{L^2(\Omega)}
\|q\|_{L^\infty(0,T)}
\right),$$
we obtain
$$
\|L_\mu u\|_{L^2(0,T;\mathcal H^{-1}(\Omega))}
\le
C
\left(
\|u_0\|_{\mathcal H^1(\Omega)}
+
\|u_1\|_{L^2(\Omega)}
+
\|p\|_{L^2(\Omega)}
\|q\|_{L^\infty(0,T)}
\right),
$$
On the other hand, the embedding $L^2(\Omega)\hookrightarrow\mathcal H^{-1}(\Omega)$ implies 
$$
\|f\|_{L^2(0,T;\mathcal H^{-1}(\Omega))}
\le
C
\|f\|_{L^2(0,T;L^2)}
\le
C
\|p\|_{L^2(\Omega)}
\|q\|_{L^\infty(0,T)}.
$$
Therefore, using ${}_0^C D_t^\alpha u = f-L_\mu u$, 
we conclude that
$$
\begin{aligned}
\|{}_0^C D_t^\alpha u\|_{L^2(0,T;\mathcal H^{-1}(\Omega))}
&\le
\|f\|_{L^2(0,T;\mathcal H^{-1}(\Omega))}
+
\|L_\mu u\|_{L^2(0,T;\mathcal H^{-1}(\Omega))}
\\
&\le
C
\left(
\|u_0\|_{\mathcal H^1(\Omega)}
+
\|u_1\|_{L^2(\Omega)}
+
\|p\|_{L^2(\Omega)}
\|q\|_{L^\infty(0,T)}
\right).
\end{aligned}
$$

The preceding estimates together with $\|f\|_{L^\infty(0,T;L^2)}=\|p\|_{L^2(\Omega)}\|q\|_{L^\infty(0,T)}$
yield \eqref{eq:apriori-strong}.

It remains to verify that the function constructed above is indeed a weak solution. For every $n$, the coefficient $u_n$ defined by
\eqref{eq:modal-repr} satisfies
$$
{}_0^C D_t^\alpha u_n+\lambda_nu_n=f_n.
$$
Let $v\in\mathcal H^1(\Omega)$ and set
$$
v^{(N)}
:=
\sum_{n=1}^{N}(v,\phi_n)_{L^2(\Omega)}\phi_n.
$$
Then
$$
v^{(N)}\to v
\quad\text{in }\mathcal H^1(\Omega),
$$
because
$$
\|v-v^{(N)}\|_{\mathcal H^1}^2
=
\sum_{n=N+1}^{\infty}
\lambda_n |(v,\phi_n)_{L^2(\Omega)}|^2
\to0.
$$
For each $N$, the modal identities give
$$
\left\langle
{}_0^CD_t^\alpha u(t),v^{(N)}
\right\rangle_{\mathcal H^{-1},\mathcal H^1}
+
a_\mu( u(t),v^{(N)})
=
(f(t),v^{(N)})_{L^2(\Omega)}.
$$
Since ${}_0^CD_t^\alpha u(t)\in\mathcal H^{-1}(\Omega)$, $u(t)\in\mathcal H^1(\Omega)$, $f(t)\in L^2(\Omega)$, we have
$$
\begin{aligned}
\left|
\left\langle
{}_0^CD_t^\alpha u(t),v^{(N)}-v
\right\rangle_{\mathcal H^{-1},\mathcal H^1}
\right|
&\le
\|{}_0^CD_t^\alpha u(t)\|_{\mathcal H^{-1}}
\|v^{(N)}-v\|_{\mathcal H^1},
\\
|a_\mu( u(t),v^{(N)}-v)|
&\le
C\| u(t)\|_{\mathcal H^1}
\|v^{(N)}-v\|_{\mathcal H^1},
\\
|(f(t),v^{(N)}-v)_{L^2}|
&\le
C\|f(t)\|_{L^2}
\|v^{(N)}-v\|_{\mathcal H^1}.
\end{aligned}
$$
All three right-hand sides tend to zero as $N\to\infty$. Passing to the
limit in the identity for $v^{(N)}$, we obtain
$$
\left\langle
{}_0^CD_t^\alpha u(t),v
\right\rangle_{\mathcal H^{-1},\mathcal H^1}
+
a_\mu( u(t),v)
=
(f(t),v)_{L^2(\Omega)}.
$$
Hence the variational identity holds for a.e. $t\in(0,T)$ and all
$v\in\mathcal H^1(\Omega)$.

Finally, suppose that $u^{(1)}$ and $u^{(2)}$ are two weak solutions
with the same data. Their difference $w=u^{(1)}-u^{(2)}$ satisfies
$$
{}_0^C D_t^\alpha w+L_\mu w=0,
\quad
w(0)=0,
\quad
\partial_tw(0)=0.
$$
Taking the duality pairing with $\phi_n$ gives
$$
{}_0^CD_t^\alpha w_n(t)+\lambda_n w_n(t)=0,
\quad
w_n(0)=0,
\quad
w_n'(0)=0,
$$
where $w_n(t):=(w(t),\phi_n)_{L^2}$. The corresponding scalar fractional initial-value problem has only
the zero solution, and hence
$$
w_n(t)=0
\quad
\text{for every }n.
$$
Completeness of $\{\phi_n\}_{n\ge1}$ in $L^2(\Omega)$ then yields
$w\equiv0$. The solution is therefore unique. 
\end{proof}

Theorem \ref{thm:wellposed} shows in particular that the terminal state
$u(\cdot,T)$ is well defined in $\mathcal H^1(\Omega)$ for every
$p\in L^2(\Omega)$. Hence the forward analysis provides the
well-defined source-to-terminal mapping required in the inverse
problem considered below.

\section{Inverse Source Problem and Uniqueness}

We now turn to the identification of the unknown spatial source factor
$p$ from the terminal observation. Throughout this section, the initial
data $u_0$ and $u_1$, as well as the temporal factor
$q\in L^\infty(0,T)$, are assumed to be known.

For each $p\in L^2(\Omega)$, let $u[p]$ denote the unique weak solution
of \eqref{eq:main}. Since the initial data are generally nonzero, the
mapping
$$
p\longmapsto u[p](\cdot,T)
$$
is affine rather than linear. To isolate the contribution of the
unknown source, let $u[0]$ denote the solution corresponding to
$p=0$ with the same initial data, and define the source-to-terminal
operator
\begin{equation}\label{eq:def-A}
\mathcal A p
:=
u[p](\cdot,T)-u[0](\cdot,T).
\end{equation}
By the linearity of the forward equation with respect to the source
term, $\mathcal A$ is a linear operator from $L^2(\Omega)$ into
$L^2(\Omega)$. Moreover, the well-posedness result established in the
previous section ensures that $\mathcal A$ is well defined.

Accordingly, setting
\begin{equation}\label{eq:def-data}
\omega
:=
\omega_{\mathrm{obs}}-u[0](\cdot,T),
\end{equation}
the inverse source problem can be written as
\begin{equation}\label{eq:inverse-operator}
\mathcal A p=\omega.
\end{equation}
Thus, the inverse problem consists in determining
$p\in L^2(\Omega)$ from the terminal source-induced response
$\omega$.

Let $\{(\lambda_n,\phi_n)\}_{n\ge1}$ be the eigensystem of $L_\mu$
introduced in the previous section, and write
$$
p_n:=(p,\phi_n)_{L^2(\Omega)}.
$$
By the spectral representation of the forward solution, the component
generated by the source $p(x)q(t)$ is given by
$$
u[p](t)-u[0](t)
=
\sum_{n=1}^{\infty}
p_n
\left[
\int_0^t
(t-s)^{\alpha-1}
E_{\alpha,\alpha}
\bigl(-\lambda_n(t-s)^\alpha\bigr)
q(s)\,ds
\right]
\phi_n .
$$
Evaluating this expression at $t=T$, we obtain
\begin{equation}\label{eq:A-spectral}
\mathcal A p
=
\sum_{n=1}^{\infty}
K_n(q)p_n\phi_n,
\end{equation}
where
\begin{equation}\label{eq:Kn}
K_n(q)
:=
\int_0^T
(T-s)^{\alpha-1}
E_{\alpha,\alpha}
\bigl(-\lambda_n(T-s)^\alpha\bigr)
q(s)\,ds .
\end{equation}

For every fixed $n$, the quantity $K_n(q)$ is well defined. Indeed,
using the standard Mittag--Leffler estimate
$$
\left|
E_{\alpha,\alpha}(-r)
\right|
\le
\frac{C}{1+r},
\quad r\ge0,
$$
we have
$$
\begin{aligned}
|K_n(q)|
&\le
C\|q\|_{L^\infty(0,T)}
\int_0^T
\frac{(T-s)^{\alpha-1}}
{1+\lambda_n(T-s)^\alpha}\,ds<
\infty.
\end{aligned}
$$

Formula \eqref{eq:A-spectral} shows that the terminal observation
determines the $n$th Fourier coefficient of the source through the
multiplier $K_n(q)$. This motivates the following modal
nondegeneracy condition:
\begin{equation}\label{eq:nondegeneracy}
K_n(q)\neq0,
\quad
\forall n\ge1.
\end{equation}

\begin{theorem}[Uniqueness from inverse problem]\label{thm:uniq}
Let $1<\alpha<2$ and $\mu<\mu^\ast=1/4$. Assume that
$$
u_0\in\mathcal H^1(\Omega),
\quad
u_1\in L^2(\Omega),
\quad
q\in L^\infty(0,T)
$$
are known, and suppose that the modal nondegeneracy condition
\eqref{eq:nondegeneracy} holds. Then the inverse problem
\eqref{eq:inverse-operator} has at most one solution
$p\in L^2(\Omega)$.
\end{theorem}

\begin{proof}
Suppose that $p_1,p_2\in L^2(\Omega)$ produce the same terminal
observation, i.e. $u[p_1](\cdot,T)=u[p_2](\cdot,T)$. Hence, by the definition of $\mathcal A$,
$$
\mathcal A(p_1-p_2)=0.
$$
Set $p:=p_1-p_2$ and denote $p_n:=(p,\phi_n)_{L^2(\Omega)}$. By the spectral representation \eqref{eq:A-spectral},
$$
\mathcal A p
=
\sum_{n=1}^{\infty}
K_n(q)p_n\phi_n.
$$
Since $\mathcal A p=0$ in $L^2(\Omega)$ and
$\{\phi_n\}_{n\ge1}$ is an orthonormal basis of $L^2(\Omega)$, we
obtain
$$
K_n(q)p_n=0,
\quad
\forall n\ge1.
$$
The assumption $K_n(q)\neq0$ for every $n$ therefore implies
$$
p_n=0,
\quad
\forall n\ge1.
$$
By the completeness of $\{\phi_n\}_{n\ge1}$ in $L^2(\Omega)$, it
follows that $p=0$ in $L^2(\Omega)$. Consequently,
$$
p_1=p_2,
$$
which proves the uniqueness.
\end{proof}

\begin{remark}
Condition \eqref{eq:nondegeneracy} has a direct spectral
interpretation. The quantity $K_n(q)$ represents the response of the
$n$th spatial eigenmode to the prescribed temporal excitation $q$ at
the observation time $T$. If
$$
K_{n_0}(q)=0
$$
for some $n_0$, then
$$
\mathcal A\phi_{n_0}=0,
$$
and therefore the component of the source along $\phi_{n_0}$ cannot
be detected from the terminal measurement. Hence uniqueness fails in
that mode. Therefore, the condition $K_n(q)\neq0$ should be
verified for the prescribed temporal excitation rather than inferred
solely from a sign assumption on $q$.
\end{remark}

\begin{corollary}[Uniqueness for $q\equiv 1$]
\label{cor:constant-q}
Let $1<\alpha<2$ and $\mu<\mu^\ast=1/4$, and assume that $q(t)\equiv 1$ for $t\in(0,T)$. For each $n\ge1$, define
$$
K_n(1)
=
\int_0^T
(T-s)^{\alpha-1}
E_{\alpha,\alpha}
\bigl(-\lambda_n(T-s)^\alpha\bigr)\,ds.
$$
Then
$$
K_n(1)
=
\frac{1-E_{\alpha,1}(-\lambda_nT^\alpha)}{\lambda_n}.
$$

Let $\mathcal Z_\alpha := \left\{\rho>0: E_{\alpha,1}(-\rho)=1 \right\}$. Since the function
$$
\rho\mapsto E_{\alpha,1}(-\rho)-1
$$
is real analytic and not identically zero, the set
$\mathcal Z_\alpha$ is at most countable. Define the exceptional set
of observation times by
$$
\mathcal T_{\mathrm{bad}}
:=
\bigcup_{n=1}^{\infty}
\left\{
\left(\frac{\rho}{\lambda_n}\right)^{1/\alpha}
:
\rho\in\mathcal Z_\alpha
\right\}.
$$
Then $\mathcal T_{\mathrm{bad}}$ is at most countable, and for every $T\in(0,\infty)\setminus\mathcal T_{\mathrm{bad}}$, one has
$$
K_n(1)\neq0,
\quad
\forall n\ge1.
$$
Consequently, for every such observation time $T$, the inverse source
problem admits at most one solution $p\in L^2(\Omega)$.
\end{corollary}

\begin{proof}
For $q\equiv1$, the identity
$$
\frac{d}{dt}
E_{\alpha,1}(-\lambda_n t^\alpha)
=
-\lambda_n t^{\alpha-1}
E_{\alpha,\alpha}(-\lambda_n t^\alpha)
$$
gives
$$
K_n(1)
=
\frac{1-E_{\alpha,1}(-\lambda_nT^\alpha)}{\lambda_n}.
$$
Hence $K_n(1)=0$ if and only if $\lambda_nT^\alpha\in\mathcal Z_\alpha$. 
Equivalently,
$$
T
=
\left(\frac{\rho}{\lambda_n}\right)^{1/\alpha}
$$
for some $\rho\in\mathcal Z_\alpha$.

Because $E_{\alpha,1}$ is an entire function,
$E_{\alpha,1}(-\rho)-1$ is real analytic on $(0,\infty)$. Moreover,
it is not identically zero. Therefore, its positive zeros are isolated,
and $\mathcal Z_\alpha$ is at most countable. Since the spectrum
$\{\lambda_n\}_{n\ge1}$ is also countable, the set
$\mathcal T_{\mathrm{bad}}$ is at most countable.

Thus, for every $T\notin\mathcal T_{\mathrm{bad}}$, we have
$$
K_n(1)\neq0
\quad
\forall n\ge1.
$$
The conclusion then follows directly from
Theorem~\ref{thm:uniq}.
\end{proof}

\begin{remark}
Suppose that $q$ is real analytic on $(0,T_{\max})$, extends
continuously to $t=0$, and satisfies $q(0)\neq0$. For each fixed $n$, define
$$
K_n(q;T)
:=
\int_0^T
(T-s)^{\alpha-1}
E_{\alpha,\alpha}
\bigl(-\lambda_n(T-s)^\alpha\bigr)
q(s)\,ds,
\quad
0<T<T_{\max}.
$$
Then
$$
K_n(q;T)
=
\frac{q(0)}{\Gamma(\alpha+1)}T^\alpha
+
o(T^\alpha)
\quad
\text{as }T\to0^+,
$$
and hence $K_n(q;\cdot)$ is not identically zero. If
$T\mapsto K_n(q;T)$ is real analytic on $(0,T_{\max})$, its zeros are
isolated. Consequently, the exceptional set
$$
\bigcup_{n=1}^{\infty}
\left\{
T\in(0,T_{\max}):K_n(q;T)=0
\right\}
$$
is at most countable. Thus, outside this exceptional set, the modal
nondegeneracy condition holds for every $n$.
\end{remark}

\section{Tikhonov regularization and conjugate gradient reconstruction}

In this section, we reformulate the inverse source problem as a
Tikhonov-regularized optimization problem and derive a continuous adjoint-based
conjugate gradient method. The adjoint equation is formulated in terms of the
right-sided Riemann--Liouville fractional derivative, which is the natural
counterpart of the left-sided Caputo derivative appearing in the forward
problem.

\subsection{Operator formulation}

Recall the linear source-to-terminal operator introduced in the previous
section:
$$
\mathcal A:L^2(\Omega)\to L^2(\Omega),
\quad
\mathcal A p
:=
u[p](\cdot,T)-u[0](\cdot,T),
$$
where $u[p]$ denotes the weak solution of the forward problem
\eqref{eq:main} corresponding to the source factor $p$, and $u[0]$
denotes the solution with zero source and the same initial data. Introducing the source-induced terminal datum
$$
\omega
:=
\omega_{\mathrm{obs}}-u[0](\cdot,T),
$$
the inverse source problem takes the linear operator form
\begin{equation}\label{eq:inverse-operator-tikh}
\mathcal A p=\omega.
\end{equation}

By the spectral representation established in the previous section,
$\mathcal A$ can be written as
\begin{equation}\label{eq:A-operator-repr}
\mathcal A p
=
\int_0^T
(T-s)^{\alpha-1}
E_{\alpha,\alpha}
\bigl(-(T-s)^\alpha L_\mu\bigr)
p\,q(s)\,ds .
\end{equation}
Equivalently,
$$
\mathcal A p
=
\sum_{n=1}^{\infty}
K_n(q)p_n\phi_n,
\quad
p_n=(p,\phi_n)_{L^2(\Omega)},
$$
where $K_n(q)$ is defined by \eqref{eq:Kn}.

We next establish the compactness of $\mathcal A$. From the
Mittag--Leffler estimate obtained in the forward analysis,
$$
\left\|
\tau^{\alpha-1}
E_{\alpha,\alpha}(-\tau^\alpha L_\mu)z
\right\|_{\mathcal H^1(\Omega)}
\le
C\tau^{\alpha/2-1}
\|z\|_{L^2(\Omega)},
\quad
0<\tau\le T.
$$
Hence, for $p\in L^2(\Omega)$,
$$
\begin{aligned}
\|\mathcal A p\|_{\mathcal H^1(\Omega)}
&\le
C
\int_0^T
(T-s)^{\alpha/2-1}
|q(s)|
\|p\|_{L^2(\Omega)}
\,ds
\\
&\le
C
\|q\|_{L^\infty(0,T)}
\|p\|_{L^2(\Omega)}
\int_0^T
(T-s)^{\alpha/2-1}\,ds
\\
&\le
C_T
\|q\|_{L^\infty(0,T)}
\|p\|_{L^2(\Omega)}.
\end{aligned}
$$
Therefore, $\mathcal A: L^2(\Omega) \longrightarrow \mathcal H^1(\Omega)$ is bounded. Since the embedding $\mathcal H^1(\Omega) \hookrightarrow L^2(\Omega)$ is compact, it follows that
$$
\mathcal A:
L^2(\Omega)
\longrightarrow
L^2(\Omega)
$$
is compact.

Consequently, although uniqueness holds under the modal
nondegeneracy condition established in the previous section, the
recovery of $p$ from perturbed terminal data is unstable. This
motivates the introduction of a regularization method.

\subsection{Tikhonov functional}

Let $\omega_{\mathrm{obs}}^\delta\in L^2(\Omega)$ be noisy terminal data satisfying
$$
\|\omega_{\mathrm{obs}}^\delta-\omega_{\mathrm{obs}}\|_{L^2(\Omega)}
\le \delta.
$$
As in the previous subsection, we subtract the known zero-source
contribution and define
$$
\omega^\delta
:=
\omega_{\mathrm{obs}}^\delta-u[0](\cdot,T).
$$
Since $\omega=\omega_{\mathrm{obs}}-u[0](\cdot,T)$, the noise level is preserved:
$$
\|\omega^\delta-\omega\|_{L^2(\Omega)}
=
\|\omega_{\mathrm{obs}}^\delta-\omega_{\mathrm{obs}}\|_{L^2(\Omega)}
\le \delta.
$$

For $\lambda>0$, we define the Tikhonov functional
\begin{equation}\label{eq:tikhonov-functional}
\mathcal J_\lambda(p)
:=
\frac{1}{2}
\left\|\mathcal A p-\omega^\delta\right\|_{L^2(\Omega)}^2
+
\frac{\lambda}{2}
\left\|p\right\|_{L^2(\Omega)}^2,
\quad
p\in L^2(\Omega).
\end{equation}
Equivalently, in terms of the forward solution,
$$
\mathcal J_\lambda(p)
=
\frac{1}{2}
\left\|u[p](\cdot,T)-\omega_{\mathrm{obs}}^\delta\right\|_{L^2(\Omega)}^2
+
\frac{\lambda}{2}
\left\|p\right\|_{L^2(\Omega)}^2.
$$

The regularized inverse problem is to determine $p_\lambda^\delta\in L^2(\Omega)$ such that
$$
\mathcal J_\lambda(p_\lambda^\delta)
=
\min_{p\in L^2(\Omega)}
\mathcal J_\lambda(p).
$$
Since $\mathcal A:L^2(\Omega)\to L^2(\Omega)$ is bounded and linear, the mapping
$$
p\mapsto
\frac12
\left\|\mathcal A p-\omega^\delta\right\|_{L^2(\Omega)}^2
$$
is continuous and convex. Moreover, because $\lambda>0$,
$$
\mathcal J_\lambda(p)
\ge
\frac{\lambda}{2}\left\|p\right\|_{L^2(\Omega)}^2,
$$
and hence $\mathcal J_\lambda$ is coercive on $L^2(\Omega)$.
The quadratic penalty term also makes $\mathcal J_\lambda$ strictly
convex. Therefore, $\mathcal J_\lambda$ admits a unique minimizer
$p_\lambda^\delta\in L^2(\Omega)$.

For later use, we define the residual
\begin{equation}\label{eq:residual}
r[p]
:=
\mathcal A p-\omega^\delta.
\end{equation}
Equivalently,
$$
r[p]
=
u[p](\cdot,T)-\omega_{\mathrm{obs}}^\delta.
$$

\subsection{Sensitivity equation and directional derivative}

Let $p\in L^2(\Omega)$ be fixed and let
$h\in L^2(\Omega)$ be an arbitrary direction. Since the forward
problem is linear with respect to the source term, for any
$\varepsilon\in\mathbb R$ we have
$$
u[p+\varepsilon h]
=
u[p]+\varepsilon\eta[h],
$$
where $\eta[h]$ is the unique weak solution of
\begin{equation}\label{eq:sensitivity-h}
\begin{cases}
{}_0^C D_t^\alpha\eta+L_\mu\eta=h(x)q(t),
& (x,t)\in\Omega\times(0,T),\\[1mm]
\eta(\cdot,0)=0,\quad
\partial_t\eta(\cdot,0)=0,
& x\in\Omega,\\[1mm]
\eta(0,t)=\eta(1,t)=0,
& t\in(0,T).
\end{cases}
\end{equation}
Thus, $\eta[h]$ represents the variation of the state generated by
the perturbation $h$.

Since the source-to-terminal operator $\mathcal A$ is linear, its
Fr\'echet derivative is independent of $p$ and satisfies
$$
\mathcal A'(p)h
=
\mathcal A h
=
\eta[h](\cdot,T).
$$
Recalling the residual
$$
r[p]=\mathcal A p-\omega^\delta,
$$
the directional derivative of the Tikhonov functional
\eqref{eq:tikhonov-functional} in the direction $h$ is
\begin{equation}\label{eq:J-directional}
\begin{aligned}
\mathcal J_\lambda'(p)h
&=
\bigl(
\mathcal A p-\omega^\delta,
\mathcal A h
\bigr)_{L^2(\Omega)}
+
\lambda(p,h)_{L^2(\Omega)}
\\
&=
\bigl(
r[p],
\eta[h](\cdot,T)
\bigr)_{L^2(\Omega)}
+
\lambda(p,h)_{L^2(\Omega)}.
\end{aligned}
\end{equation}

To obtain an explicit $L^2(\Omega)$-gradient, it remains to rewrite
the terminal term
$$
\bigl(
r[p],
\eta[h](\cdot,T)
\bigr)_{L^2(\Omega)}
$$
as an integral involving $h$. This will be achieved through an
adjoint problem involving right-sided fractional operators.

\subsection{Right-sided fractional operators}

To formulate the adjoint problem, we introduce the right-sided
Riemann--Liouville fractional operators. For $\beta>0$, the
right-sided fractional integral is defined by
\begin{equation}\label{eq:right-RL-integral}
(I_{T-}^{\beta}\psi)(t)
:=
\frac{1}{\Gamma(\beta)}
\int_t^T
(s-t)^{\beta-1}\psi(s)\,ds.
\end{equation}
For $1<\alpha<2$, the right-sided Riemann--Liouville derivative of
order $\alpha$ is defined by
\begin{equation}\label{eq:right-RL-alpha}
{}_tD_T^\alpha\psi(t)
:=
\frac{d^2}{dt^2}
\bigl(I_{T-}^{2-\alpha}\psi\bigr)(t),
\end{equation}
whereas the right-sided derivative of order $\alpha-1\in(0,1)$ is
given by
\begin{equation}\label{eq:right-RL-alpha-minus-one}
{}_tD_T^{\alpha-1}\psi(t)
:=
-\frac{d}{dt}
\bigl(I_{T-}^{2-\alpha}\psi\bigr)(t).
\end{equation}

These operators arise naturally from the fractional integration-by-parts
formula associated with the left-sided Caputo derivative and will be
used below to derive the adjoint equation and the gradient of
$\mathcal J_\lambda$.

\subsection{Derivation of the adjoint problem}

Let $\eta=\eta[h]$ be the solution of the sensitivity problem
\eqref{eq:sensitivity-h}. Formally testing the sensitivity equation by
an adjoint variable $\psi$ and integrating over $(0,T)$ yield
$$
\int_0^T
\left\langle
{}_0^C D_t^\alpha\eta(t),\psi(t)
\right\rangle\,dt
+
\int_0^T
a_\mu(\eta(t),\psi(t))\,dt
=
\int_0^T
(hq(t),\psi(t))_{L^2(\Omega)}\,dt .
$$
Since $L_\mu$ is the positive self-adjoint operator associated with the
symmetric form $a_\mu$, we may write, in the corresponding weak sense,
$$
a_\mu(\eta,\psi)
=
\langle L_\mu\eta,\psi\rangle
=
\langle \eta,L_\mu\psi\rangle .
$$

We next use the fractional integration-by-parts formula. For
$1<\alpha<2$, the left-sided Caputo derivative satisfies
\begin{equation}\label{eq:fractional-ibp}
\begin{aligned}
\int_0^T
\left\langle
{}_0^C D_t^\alpha\eta(t),\psi(t)
\right\rangle\,dt
&=
\int_0^T
\left\langle
\eta(t),{}_tD_T^\alpha\psi(t)
\right\rangle\,dt
\\
&\quad+
\left[
\left\langle
\partial_t\eta(t),
I_{T-}^{2-\alpha}\psi(t)
\right\rangle
\right]_{t=0}^{t=T}
\\
&\quad+
\left[
\left\langle
\eta(t),
{}_tD_T^{\alpha-1}\psi(t)
\right\rangle
\right]_{t=0}^{t=T}.
\end{aligned}
\end{equation}
Since $\eta(\cdot,0)=0$, $\partial_t\eta(\cdot,0)=0$, the contributions at $t=0$ vanish, and hence
$$
\begin{aligned}
\int_0^T
\left\langle
{}_0^C D_t^\alpha\eta(t),\psi(t)
\right\rangle\,dt
&=
\int_0^T
\left\langle
\eta(t),{}_tD_T^\alpha\psi(t)
\right\rangle\,dt
\\
&\quad+
\left\langle
\partial_t\eta(\cdot,T),
\left(I_{T-}^{2-\alpha}\psi\right)(T^-)
\right\rangle
\\
&\quad+
\left\langle
\eta(\cdot,T),
\left({}_tD_T^{\alpha-1}\psi\right)(T^-)
\right\rangle ,
\end{aligned}
$$
where the terminal quantities are understood as fractional traces from
the left. Substituting this identity into the weak formulation of the sensitivity
equation gives
$$
\begin{aligned}
\int_0^T
(hq(t),\psi(t))_{L^2(\Omega)}\,dt
&=
\int_0^T
\left\langle
\eta(t),
{}_tD_T^\alpha\psi(t)+L_\mu\psi(t)
\right\rangle\,dt
\\
&\quad+
\left\langle
\partial_t\eta(\cdot,T),
\left(I_{T-}^{2-\alpha}\psi\right)(T^-)
\right\rangle
\\
&\quad+
\left\langle
\eta(\cdot,T),
\left({}_tD_T^{\alpha-1}\psi\right)(T^-)
\right\rangle .
\end{aligned}
$$

To recover the terminal residual term $\bigl(r[p],\eta[h](\cdot,T)\bigr)_{L^2(\Omega)}$, we choose $\psi$ as the solution of the backward fractional problem
$$
{}_tD_T^\alpha\psi+L_\mu\psi=0
$$
subject to the fractional terminal conditions
$$
\left(I_{T-}^{2-\alpha}\psi\right)(T^-)=0,
\quad
\left({}_tD_T^{\alpha-1}\psi\right)(T^-)
=
r[p].
$$
With this choice, the space--time term involving $\eta$ vanishes, as
does the term involving $\partial_t\eta(\cdot,T)$. Consequently,
\begin{equation}\label{eq:duality-identity}
\bigl(r[p],\eta[h](\cdot,T)\bigr)_{L^2(\Omega)}
=
\int_0^T
(hq(t),\psi(t))_{L^2(\Omega)}\,dt .
\end{equation}

The adjoint problem is therefore given by
\begin{equation}\label{eq:adjoint-problem}
\begin{cases}
{}_tD_T^\alpha\psi(x,t)
-\Delta\psi(x,t)
-\dfrac{\mu}{x^2}\psi(x,t)=0,
& (x,t)\in\Omega\times(0,T),\\[3pt]
\psi(0,t)=\psi(1,t)=0,
& t\in(0,T),\\[3pt]
\displaystyle
\lim_{t\uparrow T}
I_{T-}^{2-\alpha}\psi(\cdot,t)=0,
& x\in\Omega,\\[3pt]
\displaystyle
\lim_{t\uparrow T}
{}_tD_T^{\alpha-1}\psi(\cdot,t)
=
r[p],
& x\in\Omega,
\end{cases}
\end{equation}
where
$$
r[p]
=
\mathcal A p-\omega^\delta
=
u[p](\cdot,T)-\omega_{\mathrm{obs}}^\delta.
$$

\begin{remark}
The terminal conditions in \eqref{eq:adjoint-problem} are fractional
terminal conditions associated with the right-sided
Riemann--Liouville derivative. In particular, they should not be
replaced, without further justification, by the classical terminal
conditions
$$
\psi(\cdot,T)=r[p],
\quad
\partial_t\psi(\cdot,T)=0.
$$
The terminal residual enters the adjoint problem through the fractional
boundary terms generated by the integration-by-parts formula
\eqref{eq:fractional-ibp}.
\end{remark}

\subsection{Gradient formula and conjugate gradient reconstruction}

Using the duality identity \eqref{eq:duality-identity}, the directional
derivative \eqref{eq:J-directional} can be rewritten as
$$
\begin{aligned}
\mathcal J_\lambda'(p)h
&=
\bigl(r[p],\eta[h](\cdot,T)\bigr)_{L^2(\Omega)}
+
\lambda(p,h)_{L^2(\Omega)}
\\
&=
\int_0^T
(hq(t),\psi[p](t))_{L^2(\Omega)}\,dt
+
\lambda(p,h)_{L^2(\Omega)}
\\
&=
\left(
\int_0^T q(t)\psi[p](\cdot,t)\,dt
+
\lambda p,
h
\right)_{L^2(\Omega)} ,
\end{aligned}
$$
where $\psi[p]$ denotes the solution of the adjoint problem
\eqref{eq:adjoint-problem} corresponding to the residual
$$
r[p]
=
\mathcal A p-\omega^\delta
=
u[p](\cdot,T)-\omega_{\mathrm{obs}}^\delta .
$$
Therefore, the $L^2(\Omega)$-gradient of the Tikhonov functional is
given by
\begin{equation}\label{eq:gradient}
\nabla\mathcal J_\lambda(p)
=
\int_0^T q(t)\psi[p](\cdot,t)\,dt
+
\lambda p .
\end{equation}

\begin{theorem}[Adjoint-based gradient]\label{thm:adjoint-gradient}
Let $p\in L^2(\Omega)$, let $u[p]$ be the corresponding forward
solution, and let $\psi[p]$ solve the adjoint problem
\eqref{eq:adjoint-problem}. Then the Tikhonov functional
$\mathcal J_\lambda$ defined by \eqref{eq:tikhonov-functional} is
Fr\'echet differentiable on $L^2(\Omega)$, and its gradient is
$$
\nabla\mathcal J_\lambda(p)
=
\int_0^T q(t)\psi[p](\cdot,t)\,dt
+
\lambda p .
$$
\end{theorem}

We now describe the conjugate gradient method for minimizing
$\mathcal J_\lambda$. To compute the regularized solution, we employ an adjoint-based
conjugate gradient iteration. Let $p^0\in L^2(\Omega)$ be an initial
guess. For the current iterate $p^k$, solve the forward problem
\begin{equation}\label{eq:forward-iteration}
\begin{cases}
{}_0^C D_t^\alpha u^k+L_\mu u^k
=
p^k(x)q(t),
& (x,t)\in\Omega\times(0,T),\\[1mm]
u^k(\cdot,0)=u_0,\quad
\partial_tu^k(\cdot,0)=u_1,
& x\in\Omega,\\[1mm]
u^k(0,t)=u^k(1,t)=0,
& t\in(0,T),
\end{cases}
\end{equation}
and define
\begin{equation}\label{eq:rk}
r^k
:=
\mathcal A p^k-\omega^\delta
=
u^k(\cdot,T)-\omega_{\mathrm{obs}}^\delta .
\end{equation}

Next, solve the adjoint problem
\begin{equation}\label{eq:adjoint-iteration}
\begin{cases}
{}_tD_T^\alpha\psi^k
+L_\mu\psi^k=0,
& (x,t)\in\Omega\times(0,T),\\[1mm]
\psi^k(0,t)=\psi^k(1,t)=0,
& t\in(0,T),\\[1mm]
\displaystyle
\lim_{t\uparrow T}
I_{T-}^{2-\alpha}\psi^k(\cdot,t)=0,
& x\in\Omega,\\[1mm]
\displaystyle
\lim_{t\uparrow T}
{}_tD_T^{\alpha-1}\psi^k(\cdot,t)=r^k,
& x\in\Omega .
\end{cases}
\end{equation}
The gradient at $p^k$ is then
\begin{equation}\label{eq:gk}
g^k
:=
\nabla\mathcal J_\lambda(p^k)
=
\int_0^T q(t)\psi^k(\cdot,t)\,dt
+
\lambda p^k .
\end{equation}
For $k=0$, we set $d^0=-g^0$. For $k\ge1$, the search direction is updated by the
Fletcher--Reeves formula
\begin{equation}\label{eq:FR-direction}
d^k
=
-g^k+\beta_kd^{k-1},
\quad
\beta_k
=
\frac{\|g^k\|_{L^2(\Omega)}^2}
{\|g^{k-1}\|_{L^2(\Omega)}^2}.
\end{equation}

To determine the step length, let $\eta^k$ be the solution of
\begin{equation}\label{eq:sensitivity-direction}
\begin{cases}
{}_0^C D_t^\alpha\eta^k+L_\mu\eta^k
=
d^k(x)q(t),
& (x,t)\in\Omega\times(0,T),\\[1mm]
\eta^k(\cdot,0)=0,\quad
\partial_t\eta^k(\cdot,0)=0,
& x\in\Omega,\\[1mm]
\eta^k(0,t)=\eta^k(1,t)=0,
& t\in(0,T).
\end{cases}
\end{equation}
Based on the definition of $\mathcal A$, we have 
\begin{equation}\label{eq:A-direction}
\eta^k(\cdot,T)=\mathcal A d^k.
\end{equation}

For $p^{k+1}=p^k+\zeta_kd^k$, 
linearity of $\mathcal A$ gives
$$
r[p^k+\zeta d^k]
=
r^k+\zeta\eta^k(\cdot,T).
$$
Consequently,
$$
\mathcal J_\lambda(p^k+\zeta d^k)
=
\frac12
\|r^k+\zeta\eta^k(\cdot,T)\|_{L^2(\Omega)}^2
+
\frac{\lambda}{2}
\|p^k+\zeta d^k\|_{L^2(\Omega)}^2.
$$
Differentiating with respect to $\zeta$ yields
$$
\begin{aligned}
\frac{d}{d\zeta}
\mathcal J_\lambda(p^k+\zeta d^k)
=
\bigl(
r^k+\zeta\eta^k(\cdot,T),
\eta^k(\cdot,T)
\bigr)_{L^2(\Omega)}
+
\lambda
(p^k+\zeta d^k,d^k)_{L^2(\Omega)} .
\end{aligned}
$$
Imposing the exact line-search condition $\frac{d}{d\zeta}\mathcal J_\lambda(p^k+\zeta d^k)=0$ gives
\begin{equation}\label{eq:stepsize}
\zeta_k
=
-
\frac{
(r^k,\eta^k(\cdot,T))_{L^2(\Omega)}
+
\lambda(p^k,d^k)_{L^2(\Omega)}
}{
\|\eta^k(\cdot,T)\|_{L^2(\Omega)}^2
+
\lambda\|d^k\|_{L^2(\Omega)}^2
}.
\end{equation}
Equivalently, since
$$
(g^k,d^k)_{L^2(\Omega)}
=
(r^k,\eta^k(\cdot,T))_{L^2(\Omega)}
+
\lambda(p^k,d^k)_{L^2(\Omega)},
$$
the step length may be written as
\begin{equation}\label{eq:stepsize-gradient}
\zeta_k
=
-
\frac{
(g^k,d^k)_{L^2(\Omega)}
}{
\|\mathcal A d^k\|_{L^2(\Omega)}^2
+
\lambda\|d^k\|_{L^2(\Omega)}^2
}.
\end{equation}
Since $\lambda>0$, the denominator in
\eqref{eq:stepsize} is strictly positive whenever $d^k\neq0$.

After computing $\zeta_k$, the source and residual can be updated by
\begin{equation}\label{eq:update-p-r}
p^{k+1}
=
p^k+\zeta_kd^k,
\quad
r^{k+1}
=
r^k+\zeta_k\eta^k(\cdot,T).
\end{equation}
The second identity follows directly from the linearity of
$\mathcal A$ and avoids an additional forward solve solely for the
purpose of updating the terminal residual.

The iteration can be terminated, for example, when the relative
gradient satisfies
\begin{equation}\label{eq:gradient-stop}
\frac{\|g^k\|_{L^2(\Omega)}}
{\|g^0\|_{L^2(\Omega)}}
\le\varepsilon,
\end{equation}
for a prescribed tolerance $\varepsilon>0$. When the noise level
$\delta$ is known, the discrepancy principle may also be used:
\begin{equation}\label{eq:discrepancy-stop}
\|r^k\|_{L^2(\Omega)}
=
\|u^k(\cdot,T)-\omega_{\mathrm{obs}}^\delta\|_{L^2(\Omega)}
\le
\tau\delta,
\quad
\tau>1.
\end{equation}

The resulting reconstruction procedure is summarized in
Algorithm \ref{alg:CG}.

\begin{algorithm}[H]
\caption{Adjoint-based conjugate gradient method for source reconstruction}
\label{alg:CG}
\begin{algorithmic}[1]
\State Choose an initial guess $p^0\in L^2(\Omega)$ and set $k=0$.
\State Solve \eqref{eq:forward-iteration} for $u^0$ and compute
$$
r^0=u^0(\cdot,T)-\omega_{\mathrm{obs}}^\delta.
$$
\State Solve the adjoint problem \eqref{eq:adjoint-iteration} with
terminal residual $r^0$.
\State Compute
$$
g^0
=
\int_0^T q(t)\psi^0(\cdot,t)\,dt+\lambda p^0,
\quad
d^0=-g^0.
$$
\While{the stopping criterion is not satisfied}
    \State Solve \eqref{eq:sensitivity-direction} with source
    $d^k(x)q(t)$ and obtain $\eta^k(\cdot,T)=\mathcal A d^k$.
    \State Compute the exact step length
    $$
    \zeta_k
    =
    -
    \frac{
    (r^k,\eta^k(\cdot,T))_{L^2(\Omega)}
    +
    \lambda(p^k,d^k)_{L^2(\Omega)}
    }{
    \|\eta^k(\cdot,T)\|_{L^2(\Omega)}^2
    +
    \lambda\|d^k\|_{L^2(\Omega)}^2
    }.
    $$
    \State Update
    $$
    p^{k+1}=p^k+\zeta_kd^k,
    \quad
    r^{k+1}
    =
    r^k+\zeta_k\eta^k(\cdot,T).
    $$
    \State Solve the adjoint problem
    \eqref{eq:adjoint-iteration} with terminal residual $r^{k+1}$.
    \State Compute
    $$
    g^{k+1}
    =
    \int_0^T
    q(t)\psi^{k+1}(\cdot,t)\,dt
    +
    \lambda p^{k+1}.
    $$
    \State Compute the Fletcher--Reeves coefficient
    $$
    \beta_{k+1}
    =
    \frac{
    \|g^{k+1}\|_{L^2(\Omega)}^2
    }{
    \|g^k\|_{L^2(\Omega)}^2
    }.
    $$
    \State Update
    $$
    d^{k+1}
    =
    -g^{k+1}+\beta_{k+1}d^k.
    $$
    \State Set $k\leftarrow k+1$.
\EndWhile
\State Output $p^k$ as the reconstructed source.
\end{algorithmic}
\end{algorithm}

\section{Numerical experiments}

To verify the effectiveness of the proposed reconstruction method, two
numerical examples are presented. In the first example, a known forcing
term is introduced to generate an analytical solution. Such a known term
does not affect the inverse source analysis, since only the separable
unknown source $p(x)q(t)$ is reconstructed. 

The forward problem is discretized by the L1 scheme in time for the Caputo
fractional derivative and the standard second-order finite difference scheme
in space. The singular operator
$$
L_\mu=-\frac{d^2}{dx^2}-\frac{\mu}{x^2}
$$
is approximated on the interior grid points, and the resulting linear system
at each time step is solved by a precomputed LU decomposition.

The adjoint problem is transformed into a forward fractional problem by the
time-reversal technique. More precisely, with the change of variable
$s=T-t$, the right-sided Riemann--Liouville fractional derivative is converted
into a left-sided Caputo derivative. A lifting technique is then introduced
to handle the fractional terminal condition associated with the terminal
observation. The resulting forward fractional problem is solved using the
same L1 discretization as the original equation.

The unknown source term is reconstructed by Algorithm \ref{alg:CG}. Here, we employ the discrepancy principle as an additional stopping
criterion for noise data. Specifically, the iteration is stopped once
$$
\|r^k\|_{L^2(\Omega)}
=
\|u^k(\cdot,T)-\omega_{\mathrm{obs}}^\delta\|_{L^2(\Omega)}
\le
\tau\delta.
$$
This criterion prevents excessive iterations and avoids overfitting the noise
contained in the terminal observation. For exact data, the iteration is terminated when the relative gradient criterion
$$
\frac{
\|\nabla J_\lambda(p^k)\|_{L^2(\Omega)}
}{
\|\nabla J_\lambda(p^0)\|_{L^2(\Omega)}
}
\leq\varepsilon
$$
is satisfied.

To evaluate the accuracy of the reconstructed source, we use the relative
$L^2(\Omega)$ error defined by
$$
E_p
=
\frac{
\|p_{\mathrm{rec}}-p_{\mathrm{true}}\|_{L^2(\Omega)}
}{
\|p_{\mathrm{true}}\|_{L^2(\Omega)}
},
$$
where $p_{\mathrm{true}}$ and $p_{\mathrm{rec}}$ denote the exact and
reconstructed source functions, respectively.

When noisy observations are considered, the terminal data are generated by
multiplicative random perturbations:
$$
\omega_{\mathrm{obs}}^\delta
=
\omega_{\mathrm{obs}}
+
\hat{\delta}(2\xi-1)\omega_{\mathrm{obs}},
$$
where $\xi$ is a random vector with independent entries uniformly distributed
on $(0,1)$, and $\hat{\delta}$ represents the relative noise level. Thus, the
perturbation magnitude is proportional to the magnitude of the exact terminal
observation. In all examples, we set $T=1$ and choose the initial guess as $p^0(x)=0$.

\textbf{Example 1} We first consider a test problem with an analytically known solution $
u(x,t)=t^2x^2\sin(\pi x)$, $x\in[0,1]$, which satisfies
$$
{}_0^CD_t^\alpha u
+
L_\mu u
=
f(x,t)+p_{\mathrm{true}}(x)q(t),
$$
where 
$$
p_{\mathrm{true}}(x)
=
x^2\sin(\pi x), \quad q(t)
=
\frac{2t^{2-\alpha}}{\Gamma(3-\alpha)},
$$
and 
$$
f(x,t)
=t^2
\Big[
(\pi^2x^2-2)\sin(\pi x)
-4\pi x\cos(\pi x)
-\mu\sin(\pi x)
\Big].
$$

Figure \ref{fig01} presents the numerical inversion results for two sets of fractional order and diffusion coefficient: $(\alpha, \mu)=(1.20, 0.15)$ and $(\alpha, \mu)=(1.60,0.20)$. The reconstructions corresponding to different noise levels are displayed, showing that the proposed method can effectively recover the source profile and preserve stability against noisy terminal observations. 

\begin{figure}[H] 
\centering
\subfloat[ $\alpha=1.2,\mu=0.15$]{\includegraphics[width=0.5\textwidth]{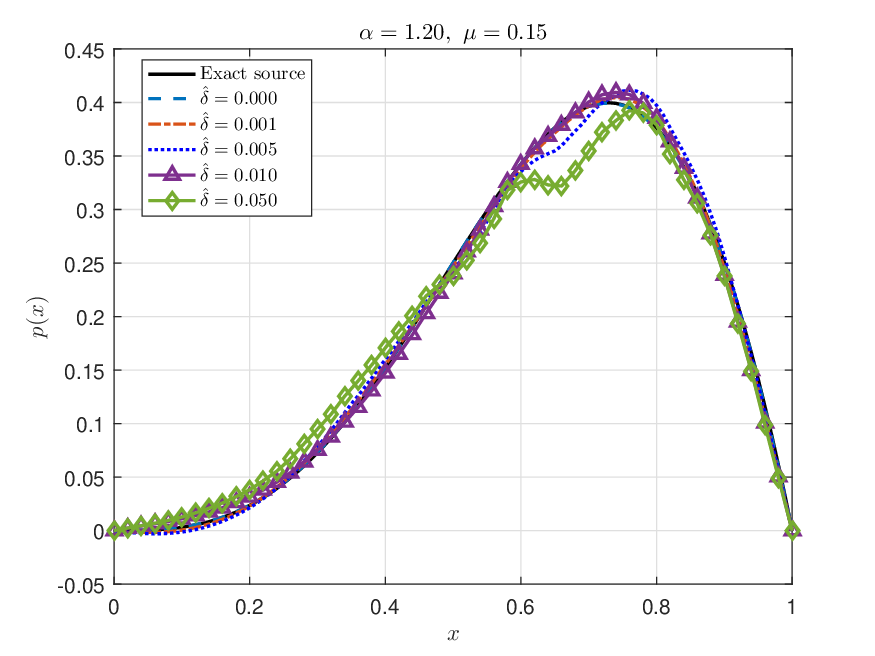}\label{fig01:sub1}}
\subfloat[ $\alpha=1.6,\mu=0.2$]{\includegraphics[width=0.5\textwidth]{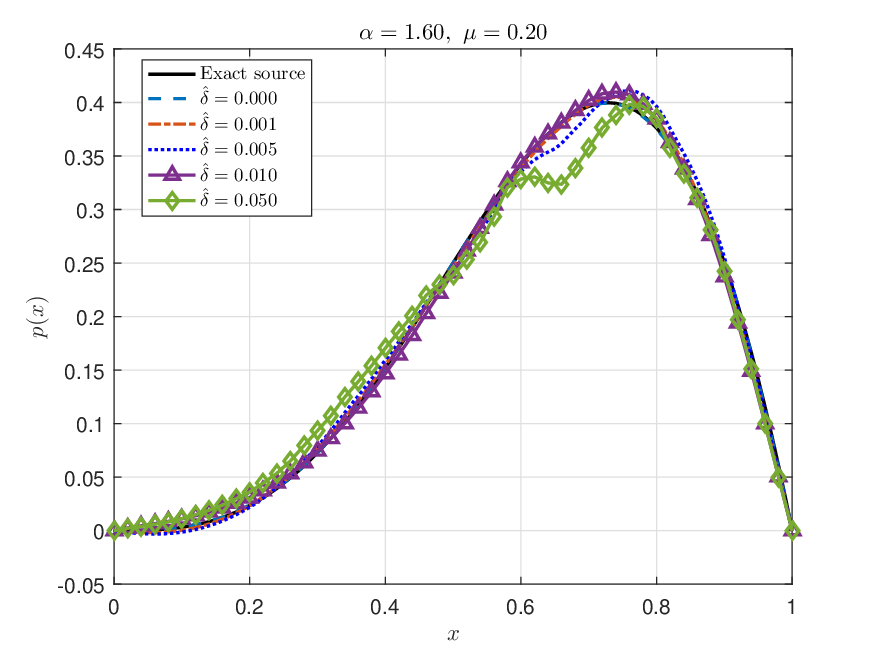}\label{fig01:sub2}}
\caption{Numerical inversions for Example 1: reconstructed source $p(x)$ versus the true solution.}\label{fig01}
\end{figure}

Table \ref{tab01} summarizes the relative reconstruction errors $E_p$ for various noise levels $\hat\delta$ alongside the corresponding regularization parameters $\lambda$. As expected, the errors increase with the noise level, yet they remain within acceptable bounds, confirming the robustness of the inversion algorithm.

\begin{table}[htbp]
\centering
\caption{Relative errors $E_p$ for different noise levels with corresponding regularization parameters in Example 1.}
\begin{tabular}{ccc|ccc}
\hline
\multicolumn{3}{c|}{$(\alpha,\mu)=(1.20,0.15)$} &
\multicolumn{3}{c}{$(\alpha,\mu)=(1.60,0.20)$} \\
\hline
$\hat{\delta}$ & $\lambda$ & $E_p$ &
$\hat{\delta}$ & $\lambda$ & $E_p$ \\
\hline
0.000 & 6.500000e-15 & 2.579009e-03 &
0.000 & 6.500000e-15 & 2.891620e-03 \\
0.001 & 6.500000e-07 & 1.828699e-02 &
0.001 & 6.500000e-07 & 1.850687e-02 \\
0.005 & 6.500000e-06 & 3.914721e-02 &
0.005 & 6.500000e-06 & 3.698948e-02 \\
0.010 & 6.500000e-05 & 3.139369e-02 &
0.010 & 6.500000e-05 & 3.158957e-02 \\
0.050 & 6.500000e-04 & 8.607583e-02 &
0.050 & 6.500000e-04 & 8.011326e-02 \\
\hline
\end{tabular}
\label{tab01}
\end{table}

\textbf{Example 2} We next consider a second test problem where the unknown spatial source term is chosen as a piecewise linear function, defined by
$$
p_{\mathrm{true}}(x) =
\begin{cases}
2x, & 0 \le x \le 0.5, \\[1mm]
2(1-x), & 0.5 < x \le 1.
\end{cases}
$$
This function belongs to $L^2(\Omega)$ but is not smooth at $x=0.5$, providing a test case for evaluating the reconstruction capability for nonsmooth source terms. The time-dependent component is simply taken as $$q(t) = 1, \quad t>0.$$ The choice $q(t)=1$ satisfies the modal nondegeneracy condition discussed in Section 4, and hence guarantees the uniqueness of the source reconstruction. The unknown field $u(x,t)$ is subject to homogeneous Dirichlet boundary conditions
$$
u(0,t) = u(1,t) = 0, \quad t>0,
$$
and zero initial conditions
$$
u(x,0) = 0, \quad \partial_t u(x,0) = 0, \quad x\in(0,1).
$$

In contrast to Example 1, this problem does not possess an analytical closed-form solution. Consequently, we generate synthetic observation data by numerically solving the forward problem using the same numerical scheme described above.  

Figure \ref{fig02} displays the reconstructed source profiles for the same two parameter sets. The results indicate that the proposed method remains capable of capturing the essential features of the nonsmooth source, even when the data are contaminated by noise.

\begin{figure}[H] 
\centering
\subfloat[ $\alpha=1.2,\mu=0.15$]{\includegraphics[width=0.5\textwidth]{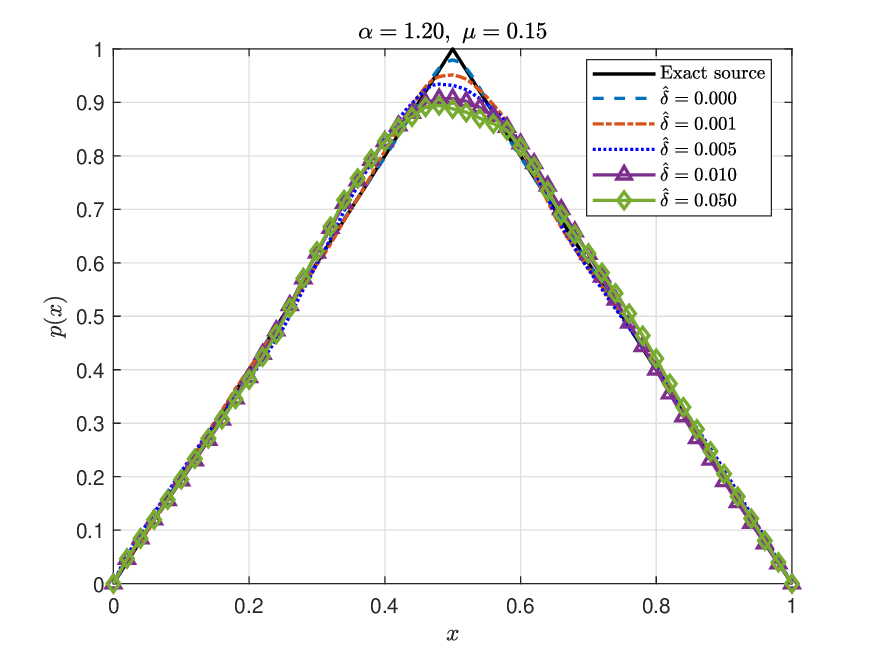}\label{fig02:sub1}}
\subfloat[ $\alpha=1.6,\mu=0.2$]{\includegraphics[width=0.5\textwidth]{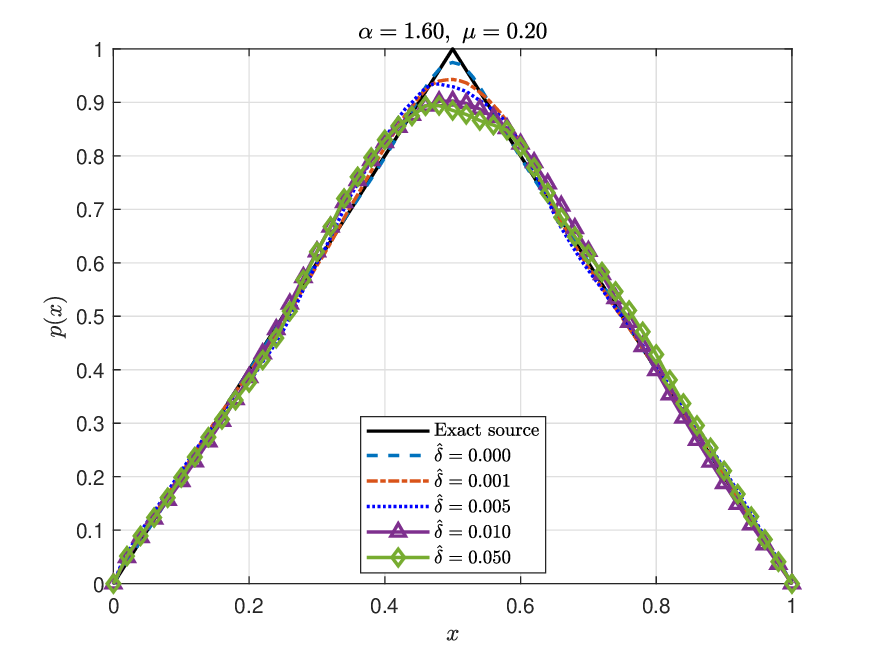}\label{fig02:sub2}}
\caption{Numerical inversions for Example 2: reconstructed source $p(x)$ versus the true piecewise linear profile.}\label{fig02}
\end{figure}

Table \ref{tab02} summarizes the reconstruction errors and regularization parameters for different noise levels in Example 2. The errors increase gradually with the noise level, as expected, yet remain at a satisfactory level. These results, together with those from Example 1, confirm that the proposed inversion algorithm is capable of handling both smooth and nonsmooth source profiles with consistent accuracy and stability.

\begin{table}[htbp]
\centering
\caption{Relative errors $E_p$ for different noise levels with corresponding regularization parameters in Example 2.}
\begin{tabular}{ccc|ccc}
\hline
\multicolumn{3}{c|}{$(\alpha,\mu)=(1.20,0.15)$} &
\multicolumn{3}{c}{$(\alpha,\mu)=(1.60,0.20)$} \\
\hline
$\hat{\delta}$ & $\lambda$ & $E_p$ &
$\hat{\delta}$ & $\lambda$ & $E_p$ \\
\hline
0.000 & 6.500000e-15 & 4.909266e-03 &
0.000 & 6.500000e-15 & 6.479352e-03 \\
0.001 & 6.500000e-09 & 1.719298e-02 &
0.001 & 6.500000e-09 & 2.098027e-02 \\
0.005 & 6.500000e-08 & 2.984921e-02 &
0.005 & 6.500000e-08 & 3.350629e-02 \\
0.010 & 6.500000e-07 & 3.900697e-02 &
0.010 & 6.500000e-07 & 4.229799e-02 \\
0.050 & 6.500000e-06 & 4.776850e-02 &
0.050 & 6.500000e-06 & 5.110032e-02 \\
\hline
\end{tabular}
\label{tab02}
\end{table}

\begin{remark}
The Hardy inequality used in the one-dimensional analysis admits a natural counterpart in higher dimensions \cite{SuD2012}. Let
$$
\Omega
=
\prod_{j=1}^{n}(0,L_j)\subset\mathbb R^n,
\quad L_j>0,
$$
and assume that the origin is a vertex of $\partial\Omega$. Then, for every $v\in H_0^1(\Omega)$,
$$
\frac{(3n-2)^2}{4}
\int_{\Omega}
\frac{|v(x)|^2}{|x|^2}\,dx
\le
\int_{\Omega}
|\nabla v(x)|^2\,dx.
$$
Hence, the spectral framework employed in the one-dimensional analysis, including the well-posedness of the direct problem, the spectral representation of the solution, and the subsequent inverse source analysis, extends to the two-dimensional rectangular setting under the above subcritical condition on $\mu<4$. This provides the theoretical basis for the two-dimensional numerical experiment presented below. The same argument also applies to rectangular domains in dimensions $n\geq 3$, with the corresponding Hardy constant $(3n-2)^2/4$. 

For comparison, when the origin is an interior point of a bounded domain $\Omega\subset\mathbb R^n$, the classical Hardy inequality for $n\geq 3$ reads 
$$
\frac{(n-2)^2}{4}
\int_\Omega
\frac{|v(x)|^2}{|x|^2}\,dx
\le
\int_\Omega |\nabla v(x)|^2\,dx,
\qquad
v\in H_0^1(\Omega),
$$
where $(n-2)^2/4$ is the optimal constant. Thus, the same analytical framework remains applicable for
$\mu<\frac{(n-2)^2}{4}$. 
In contrast, when $n=2$ and the origin is an interior point, the optimal constant in the classical Hardy inequality with weight $|x|^{-2}$ is zero. Therefore, no corresponding inequality with a positive constant holds in this case, and the above $H_0^1$-based argument does not directly apply.
\end{remark}

\textbf{Example 3.}
Let $\Omega=(0,1)\times(0,1)$. We consider
$$
{}_0^C D_t^\alpha u
-\Delta u
-\frac{\mu}{x^2+y^2}u
=
f(x,y,t)+p(x,y)q(t),
\qquad 1<\alpha<2,
$$
subject to homogeneous Dirichlet boundary conditions and zero initial data. Choose
$$
u(x,y,t)
=
t^2(x^2+y^2)\sin(\pi x)\sin(\pi y), \quad 
q(t) = \frac{2t^{2-\alpha}}{\Gamma(3-\alpha)}, 
$$
so that the exact spatial source is
$$
p^\dagger(x,y)
=
(x^2+y^2)\sin(\pi x)\sin(\pi y).
$$
The forcing term $f(x,y,t)$ is then determined by substituting the above expressions into the governing equation. The spatial source $p^\dagger(x,y)$ is reconstructed from exact and noisy terminal observations at $T=1$.

Numerical results are displayed in Figure 3 and Table 3 for two representative parameter sets: $(\alpha,\mu)=(1.20,0.15)$ and $(\alpha,\mu)=(1.60,3.3)$. The reconstructed profiles demonstrate good agreement with the true source, confirming the effectiveness of the proposed method in the two-dimensional spatial case.

\begin{table}[htbp]
\centering
\caption{Relative reconstruction errors $E_p$ for different noise levels $\hat{\delta}$ and regularization parameters $\lambda$ in Example 1.}
\begin{tabular}{ccc|ccc}
\hline
\multicolumn{3}{c|}{$(\alpha,\mu)=(1.20,0.15)$} &
\multicolumn{3}{c}{$(\alpha,\mu)=(1.60,3.30)$} \\
\hline
$\hat{\delta}$ & $\lambda$ & $E_p$ &
$\hat{\delta}$ & $\lambda$ & $E_p$ \\
\hline
0.000 & 6.500000e-15 & 2.917086e-03 &
0.000 & 6.500000e-15 & 2.758522e-03 \\
0.001 & 6.500000e-09 & 1.187930e-02 &
0.001 & 6.500000e-09 & 1.205956e-02 \\
0.005 & 6.500000e-08 & 2.578972e-02 &
0.005 & 6.500000e-08 & 2.665749e-02 \\
0.010 & 6.500000e-07 & 2.643229e-02 &
0.010 & 6.500000e-07 & 2.644666e-02 \\
0.050 & 6.500000e-06 & 6.955686e-02 &
0.050 & 6.500000e-06 & 5.326608e-02 \\
\hline
\end{tabular}
\label{tab:example1_new}
\end{table}

\begin{figure}[H] 
\centering
\subfloat[ $\alpha=1.2,\mu=0.15$]{\includegraphics[width=0.5\textwidth]{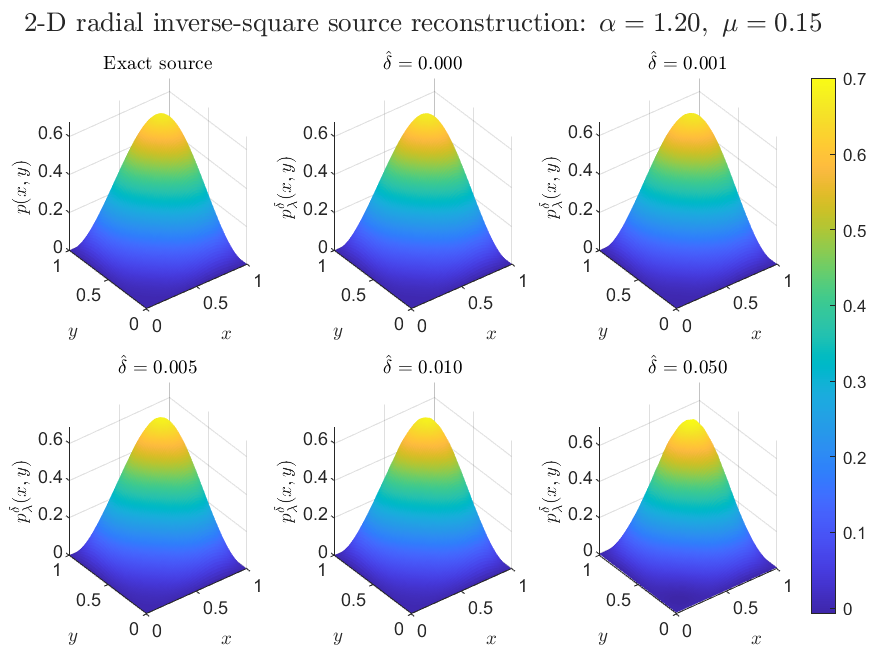}\label{fig02:sub1}}
\subfloat[ $\alpha=1.6,\mu=3.3$]{\includegraphics[width=0.5\textwidth]{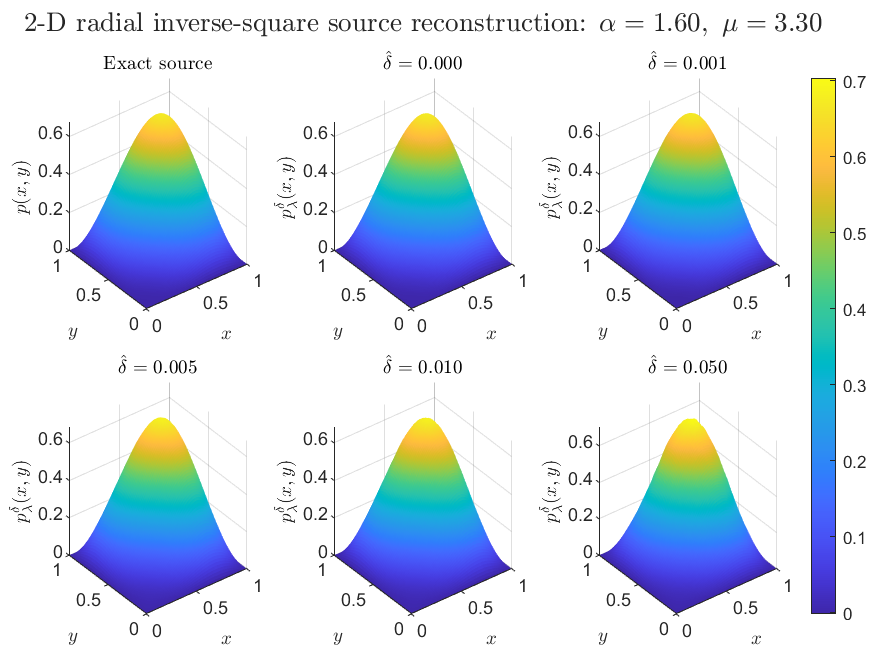}\label{fig02:sub2}}
\caption{Numerical inversions for Example 2: reconstructed source $p(x)$ versus the true piecewise linear profile.}\label{fig02}
\end{figure}

Overall, the numerical results from the three examples demonstrate that the proposed method provides accurate and stable reconstructions for both one- and two-dimensional problems, smooth and nonsmooth source profiles, and different fractional orders as well as singular potentials.

\section{Conclusion}

In this work, we investigated an inverse source problem for a time-fractional diffusion-wave equation with a singular inverse-square potential. The source term was assumed to be separable into a known temporal component and an unknown spatial component. By exploiting Hardy-type inequalities and the spectral properties of the associated singular elliptic operator, the well-posedness of the forward problem was established in an appropriate energy space.

For the inverse problem, we analyzed the terminal observation operator and proved the uniqueness of the spatial source under a modal nondegeneracy condition on the temporal factor. This condition characterizes the detectability of each eigenmode from the terminal measurement and provides a general criterion for source identifiability. To overcome the inherent ill-posedness of the reconstruction, a Tikhonov regularization strategy was developed. The Fr\'{e}chet gradient of the regularized functional was derived through an adjoint problem involving the right-sided Riemann--Liouville fractional derivative, which subsequently enabled the design of an adjoint-based conjugate gradient method with an exact line search for numerical reconstruction. The performance of the proposed method was evaluated through three numerical examples, covering both one- and two-dimensional spatial domains, smooth and nonsmooth source profiles, various fractional orders, and different singular potential parameters. Reconstruction results from both exact and noisy terminal data demonstrate that the method consistently yields accurate and stable recovery of the unknown source, confirming its effectiveness and robustness for practical inverse problems in fractional diffusion systems.

Future work will focus on extending the proposed approach to more general geometries, variable coefficients, and other types of incomplete measurements, as well as developing more efficient numerical schemes for large-scale inverse problems.

\section*{Acknowledgments}
This work is supported by National Natural Science Foundation of China (12261004), Guangdong Basic and Applied Basic Research Foundation (2025A1515012248), Innovation Team Project of Regular Universities in Guangdong Province (2025KCXTD037). The third author is supported by JSPS KAKENHI Grant Numbers JP23KK0049, JP26K06926 and FY2025 MUSUBIME of Kyoto University.

\vspace{0.5cm}

\noindent \textbf{Conflict of interest:} The authors declare that there is no conflict of interest regarding this submitted manuscript.


\begin{thebibliography}{10}

\bibitem{SchneiderWyss1989}
Schneider W R, Wyss W.
Fractional diffusion and wave equations.
Journal of Mathematical Physics, 1989, 30(1): 134-144.

\bibitem{SakamotoYamamoto2011}
Sakamoto K, Yamamoto M.
Initial value/boundary value problems for fractional diffusion-wave
equations and applications to some inverse problems.
Journal of Mathematical Analysis and Applications,
2011, 382(1): 426-447.

\bibitem{JinB2023} Jin B, Zhou Z. Numerical treatment and analysis of time-fractional
evolution equations, volume 214. Springer, 2023.
        
\bibitem{JinRundell2015}
Jin B, Rundell W.
A tutorial on inverse problems for anomalous diffusion processes.
Inverse Problems, 2015, 31(3): 035003.
                    
\bibitem{JiangD2020}
Jiang D, Liu Y, Wang D. Numerical reconstruction of the spatial component in the source term of a time-fractional diffusion equation. Advances in Computational Mathematics, 2020, 46(3): 43.


\bibitem{WangZ2023} Wang Z, Qiu S, Yu S, et al.  Exponential
tikhonov regularization method for solving an inverse source problem of time
fractional diffusion equation. Journal of Computational Mathematics, 2023, 41(2):173–190.

\bibitem{WeiT2022} Wei T, Luo Y. A generalized quasi-boundary value method for
recovering a source in a fractional diffusion-wave equation. Inverse Problems, 2022, 
38(4):045001.

\bibitem{LuoY2024}
Luo Y, Wei T. Uniqueness and numerical method for determining a spatial source term in a time-fractional diffusion wave equation[J]. Journal of Scientific Computing, 2024, 99(2): 51.

\bibitem{BarasGoldstein1984}
Baras P, Goldstein J A.
The heat equation with a singular potential.
Transactions of the American Mathematical Society,
1984, 284(1): 121-139.

\bibitem{VazquezJL2000} Vazquez J L, Zuazua E. The Hardy inequality and the asymptotic behaviour of the heat equation with an inverse-square potential. Journal of Functional Analysis, 2000, 173(1): 103-153.
    
\bibitem{SuD2012} Su D, Yang Q H. On the best constants of Hardy inequality in $\mathbb{R}^{n-k}\times
    (\mathbb{R}_+)^k$ and related improvements[J]. Journal of Mathematical Analysis and Applications, 2012, 389(1): 48-53.    

\bibitem{VancostenobleJ2011} Vancostenoble J. Lipschitz stability in inverse source problems for singular parabolic equations. Communications in Partial Differential Equations, 2011, 36(8): 1287-1317.

\bibitem{CazacuC2014} Cazacu C. Controllability of the heat equation with an inverse-square potential localized on the boundary. SIAM Journal on Control and Optimization, 2014, 52(4): 2055-2089.

\bibitem{AlaouiM2021} Alaoui M, Hajjaj A, Maniar L, Salhi J. Inverse problem for a degenerate/singular parabolic system with Neumann boundary conditions[J]. Journal of Inverse and Ill-posed Problems, 2021, 29(6): 791-821. 

\bibitem{AnhCT2022} Anh C T, Toi V M, Tuan T Q. Lipschitz stability in inverse source problems for a singular parabolic equation. Applicable Analysis, 2022, 101(8): 2805-2824. 


\bibitem{QinX2023} Qin X, Li S. Local logarithmic stability of an inverse coefficient problem for a singular heat equation with an inverse-square potential. Applicable Analysis, 2023, 102(7): 1995-2017. 

\bibitem{QinX2025} Qin X, Li S. Inverse source problem for a singular parabolic equation with variable 
coefficients. Mathematics, 2025, 13(10): 1678.
	
\bibitem{NedjmaM2025} Nedjma M, Chattouh A. Inverse source recovery in a class of singular diffusion equations via optimal control. Journal of Mathematics, Mechanics \& Computer Science, 2025, 127(3): 136.
    
\bibitem{MaChen2024}
Ma S, Chen H.
Efficient finite difference/spectral approximation for the
time-fractional diffusion equation with an inverse square potential
on the unit ball.
Computers \& Mathematics with Applications,
2024, 167: 232-238.
    

\bibitem{ChattouhA} Chattouh A, Nedjma M, Sidi H O. On the uniqueness and numerical reconstruction of spatial source term in a time-fractional diffusion model with singular inverse-square potential, Mathematical Modelling and Analysis, Accepted. 


\bibitem{KilbasAA2006} Kilbas A A, Srivastava H M, Trujillo J J. Theory and applications of fractional differential equations. Amsterdam: elsevier, 2006.


	
\end{thebibliography}

\end{document}